\documentclass[a4paper, 11pt]{amsart}

\usepackage{custom}
\usepackage{todonotes}

\newcommand{\Helvetica}[2][]{#1}

\title{A volume comparison theorem for characteristic numbers}

\thanks{The author has been supported by the DFG Research Training Group 1493 \textit{Mathematische Strukturen in der modernen Quantenphysik} and the ISF project \textit{Action now: geometry and dynamics of group actions}. He also thanks his PhD supervisor Prof.\ Dr.\ Thomas Schick for his guidance and support during the author's PhD project from which this paper developed.}

\author[D.~Luckhardt]{Daniel Luckhardt$^{*}$}
\address[D.~Luckhardt]{Department of Mathematics, Ben-Gurion University of the Negev, Israel.}
\email{luckhard@post.bgu.ac.il}

\begin{document}

\begin{abstract}
	We show that assuming lower bounds on the Ricci curvature and the injectivity radius the absolute value of certain characteristic numbers of a Riemannian manifold, including all Pontryagin and Chern numbers, is bounded proportionally to the volume.
	The proof relies on Chern-Weil theory applied to a connection constructed from Euclidean connections on charts in which the metric tensor is harmonic and has bounded H\"older norm.
	
	We generalize this theorem to a Gromov-Hausdorff closed class of rough Riemannian manifolds defined in terms of Hölder regularity.
	Assuming an additional upper Ricci curvature bound, we show that also the Euler characteristic is bounded proportionally to the volume.
	Additionally, we remark on a volume comparison theorem for Betti numbers of manifolds with an additional upper bound on sectional curvature.
	It is a consequence of a result by Bowen.
\end{abstract}

\maketitle

\section{Introduction}\label{sec:intro}

\subsection{Characteristic numbers}\label{ssec:charNumb}
Characteristic numbers are important invariants of compact oriented differentiable manifolds of even dimension $\Dim$, they play a key role for instance in cobordism theory \cites{Stong15}[26.3]{Tu17}---among them the Euler characteristic and Pontryagin numbers.
The most basic example is the Euler characteristic of a surface, which by the Gauss-Bonnet Theorem can be expressed through the integral
\[ 
	\chi (M) = \frac{1}{2\pi} \int K \diff \varvol 
\]
where $ K $ denotes the Gaussian curvature.
In general, characteristic numbers are an invariant of principal bundles \cite[32]{Tu17}.
But we will restrict ourselves to the two cases where the principle bundle is the frame bundle of the tangent bundle or the complexified tangent bundle.
The two cases correspond to the structure groups $\linGr{GL}{\Dim}[\R]$ and $\linGr{GL}{\Dim}[\C]$.
We will explain the theory mostly for the $\linGr{GL}{\Dim}[\C]$ case but everything works in parallel by replacing $\C$ with $\R$.
In each case the invariants are given by $\linGr{GL}{\Dim}[\C]$-invariant ($\linGr{GL}{\Dim}[\R]$-invariant resp.) polynomials. In both cases specific generating systems can be explicitly stated.
The invariants coming from these generators are called Chern numbers (Pontryagin numbers resp.).
The Euler characteristic, that corresponds to the Euler class---or Pfaffian---is not of this form. 
But the corresponding bound follows under slightly stricter curvature assumptions, see \cref{thm:comparison_Euler} below.

We take the perspective of Chern-Weil theory \cite[23,25.8]{Tu17}. 
We denote by $(\C[ \LieAlg{gl}{\Dim}[\C] ])^{\linGr{GL}{\Dim}[\C]}$ the subring of $\linGr{GL}{\Dim}[\C]$-invariant polynomials from the ring $ 
	\C[ \LieAlg{gl}{\Dim}[\C] ] \simeq \C[\mathrm{M}_\Dim(\C)]
$, i.e.\ polynomials $\Pi$ on the matrix ring $ \mathrm{M}_\Dim(\C) $ that are invariant under the $ \linGr{GL}{\Dim}[\C] $-action $\Pi(\blank) \mapsto \Pi( A(\blank) A^{-1} ) $ (for every $A \in \linGr{GL}{\Dim}[\C] $).
The key in this approach is the \definiendum{Chern-Weil homomorphism} 
\[
	\phi\colon (\C[ \LieAlg{gl}{\Dim}[\C] ])^{\linGr{GL}{\Dim}[\C]} \to H^*(M, \C) 
\]
($\phi\colon (\R[ \LieAlg{gl}{\Dim}[\R] ])^{\linGr{GL}{\Dim}[\R]} \to H^*(M, \R) $ resp.)
valued in the de Rham cohomology ring with $\C$-coefficients $ H^*(M, \C) $.
This ring homomorphism is defined in terms of the curvature tensor $\Riem_\nabla$ of a connection $\nabla$, e.g.\ in case of a Riemannian manifold the Levi-Civita connection $\nabla$.
More precisely the cohomology class $\phi(\Pi) $ is represented by an evaluation of $\Pi$ in the entries of the curvature tensor $\Riem_\nabla$.
The degree of $\phi(\Pi)$ is actually twice the degree of $\Pi$.
It can be shown that this map does not depend on the choice of the connection \cite[23.5,25.8]{Tu17}.
Hence, in case we consider a connection induced by choice of a curvature tensor or by choice of a Riemannian metric via its curvature tensor, it is independent of the respective choice as well.
Each \definiendum{characteristic number} is determined by some invariant polynomial $\Pi \in (\C[ \LieAlg{gl}{\Dim}[\C] ])^{\linGr{GL}{\Dim}[\C]}$ of degree $\nicefrac{d}{2}$ and defined as the evaluation
\begin{equation}\label{eq:characteristicNumb}
	\Pi[M] \coloneqq \phi(\Pi)[M] 
\end{equation}
of the de Rham form $\phi(\Pi)$
on the fundamental class $[M]$ of $M$.
The cohomology class $\phi(\Pi)$ is itself a homeomorphism invariant called \definiendum{characteristic class}.

\subsection{Main theorems and history}
Define the \definiendum{Ricci curvature on the whole} of a $d$-dimensional Riemannian manifold $M$, $ \Ricci M $, as the set of all real numbers $ \Ricci(v,v) $ where $v \in \tang_x M $ with $ \|v\|_g=1$ and $ x \in M$.
In the same fashion define the \definiendum{sectional curvature on the whole} of $M$ as $ \Sec M \coloneqq \setBuilder{\Sec(v,w)}{v,w\in \tang_x M, \langle v,w \rangle_g > 0, x \in M} $.
The main theorem is the following comparison theorem, a theorem that relates one geometric quantity to another one:
\begin{restatable}{theorem}{comparison}\label{thm:comparison}
	Let $ \Pi \in (\C[ \LieAlg{gl}{\Dim}[\C] ])^{\linGr{GL}{\Dim}[\C]} $ ($ \Pi \in (\R[ \LieAlg{gl}{\Dim}[\R] ])^{\linGr{GL}{\Dim}[\R]} $ resp.) be an invariant polynomial for some even dimension $\Dim$, $ \iota> 0 $, and $ \underline{\kappa}\in\R $.
	There is a constant $ C = C(\Pi, \iota,\underline{\kappa}) $ such that
	\[ \left|\Pi[M]\right| \leq C \vol (M) \]
	for any closed oriented $d$-dimensional Riemannian manifold with 
	\begin{assumptions}
		\item\label{assump:curv} 
			Ricci curvature bound $ \Ricci M \geq \underline{\kappa} $ and
		\item\label{assump:injRad}
			injectivity radius bound $ \injRad M \geq \iota $.
	\end{assumptions}
\end{restatable}
We will prove this theorem directly, though it is also a direct consequence of the slightly more general \cref{thm:comparison_Hold} below, which we will proved as a generalization of \cref{thm:comparison}.
The proof is actually exactly the same with the exception of one place where more work is needed (Proposition~\ref{prop:connIndep}).

There is a long history of such comparison theorems. Among them some celebrated results of Gromov: He started off with a comparison theorem for Betti numbers $b_0, \ldots, b_\Dim $ over any field of a compact connected Riemannian manifold $(M, g)$ of arbitrary dimension $\Dim$ and with sectional curvature $ \Sec M \subset [0,\infty) $ \cites{Gromov81_Betti}[12.5]{Petersen16},
that provides a bound
\[
	\sum_{ \xN = 0}^\Dim b_\xN \leq C(\Dim)
\text{;}
\]
note that we will always indicate dependents of constants, like $ C(\Dim) $, on parameters, like $d$, by parentheses.
 A few years later a comparison theorem for negatively curved real-analytic manifolds appeared \cite[Theorem~2]{BallmannGromovSchroeder85}: Let $M$ have finite volume, sectional curvature $ \Sec(M) \subset [-1,0] $, and a universal cover $ \tilde{M} $ that admits an isometric splitting of the form $ \tilde{M} \simeq N \times \R $. Then
\[
	\sum_{ \xN = 0}^\Dim b_\xN \leq C(\Dim) \cdot \vol(M)
\text{.}
\]

Around the same time \textcite{CheegerGromov85} showed that the $\eta$-invariant of a compact $(4l-1)$-dimensional manifold $M$ is bounded by
\[
	\left|\eta(M)\right| \leq C(l) \cdot \vol(M)
\]
provided that $\Sec (M) \subset [-1,1]$ and that some profinite or normal cover of $M$ has injectivity radius at least 1.
In general the $\eta$-invariant can be defined in terms of an elliptic self-adjoint operators \cite{Muller94,Bismut98}.
It is possible to relate this invariant to characteristic numbers when assuming that $M$ is the boundary of an oriented manifold $N$:
The signature $\sigma(N)$ of $N$ can be expressed as  $ 
	\sigma(N) = L_l[N] + \eta(M) + \mathrm{I\!I}_\sigma(M) 
$ where $\mathrm{I\!I}_\sigma(M)$ is an expression in the second fundamental form of $M$ and $L_l[N]$ is a characteristic number given by the $l$-th $L$-polynomial \cite[B.5]{Tu17}.

In view of a comparison theorem for Betti numbers, another statement closer related to \cref{thm:comparison} follows from a result by \textcite{Bowen15} and will be proved at the end of \cref{sec:HolderNorm}:

\begin{restatable}{theorem}{comparisonBetti}\label{thm:comparison_Betti}
	Let $\Dim = 1, 2,\ldots$, and $\iota > 0$, 
	$-\infty < \underline{\kappa} \leq \overline{\kappa} < \infty $.
	For $i = 0, \ldots, d$ there is a constant $ 
		C_i = C_i (d, \iota, \underline{\kappa}, \overline{\kappa} ) 
	$ such that 
	\[
		| b_i(M) | \leq C_i \vol (M)
	\]	
	for any $d$-dimensional closed Riemannian manifold $M$ with $\injRad(M) > \iota $, $ \Ricci(M) \geq \underline{\kappa} $, and a sectional curvature bound $ \Sec(M) \leq \overline{\kappa} $.
\end{restatable}

Further theorems in this context were recently proved by \textcite{BaderGelanderSauer20} who studied the integer homology of a $d$-dimensional complete Riemannian manifolds $M$ with $ \Sec(M) \subset [-1,0) $.
They obtained a bound on the rank by
\[
	\rank H_k(M; \Z) \leq C(d) \cdot \vol(M)
\]
for every degree $k$. Under the additional assumption $ d \neq 3 $ they also obtained the parallel bound
\[
	\operatorname{torsion} H_k(M; \Z) \leq C(d) \cdot \vol(M)
\]
on the torsion part of the homology for every degree $k$.

\subsection{The proof and outline}
We will start by giving an overview of regularity theory of Riemannian manifolds (\cref{sec:HolderNorm}).
With regard to the proof of \cref{thm:comparison} one should at first note that when strengthening \cref{assump:curv} to both sided Ricci curvature bound (and retaining \cref{assump:injRad}) the claim follows easily from a regularity result on the Riemannian metric tensor obtained by \textcite{Anderson90}.
This result provides conditions to a manifold to belong to a class $\mathcal{M}^\Dim(\Sobolev{2}[p]\leq_r^{\textnormal{harm}})$ of Sobolev regular manifolds, which will be introduced by \cref{eq:normBd_Sobolev}. 

In \cref{sec:EulerChar} we will use this to prove the announced variant for the Euler characteristic, basically by applying H\"older's inequality.

\begin{restatable}{theorem}{comparisonEuler}\label{thm:comparison_Euler}
	Let $ \Dim $ be some even dimension $\Dim$ and $ Q,r> 0 $, $p \in (d,\infty)$.
	There is a constant $ C $ such that
	\[ 
		\left|\chi (M)\right| \leq C \vol (M) 
	\]
	where
	\begin{claims}
	\item\label{claim:thm:comparison_Euler_geom}
		$C = C(\iota, \underline{\kappa}, \overline{\kappa} )$ and 
		$M$ ranges over closed oriented $\Dim$-dimensional Riemannian manifolds with $\injRad(M) \geq \iota$ and $ \Ricci M \subset [\underline{\kappa}, \overline{\kappa}] $, or
	\item\label{claim:thm:comparison_Euler_Sobolev}
		$C = C(r, p, Q)$ and 
		$M$ ranges over closed oriented $\Dim$-dimensional Riemannian manifolds in the class $\mathcal{M}^\Dim(\Sobolev{2}[p]\leq_r^{\textnormal{harm}} Q )$.
	\end{claims}
\end{restatable}

But under the weaker \cref{assump:curv,assump:injRad} there is only a bound on the $\Sobolev{1}[p]$-norm and consequently, as we will explain in \cref{ssec:chartNorms}, on the $\Hoeld{}[\alpha]$-norm.
Having bounded $\Hoeld{}[\alpha]$-norm can be roughly described as that there is an atlas of $M$ consisting of charts of controlled size in which the metric tensor is $\Hoeld{}[\alpha]$-bounded. The trick is then to define a connection, we call \textit{piecewise Euclidean connection}, from which we can define a curvature tensor (\cref{sec:piecewiseEuclConn}). As characteristic classes do not depend on the connection, we can use this curvature tensor instead and conclude the proof (\cref{sec:comparisonThm}).

\Cref{sec:comparison_C1alpha} exploits the feature of this piecewise Euclidean connection that it can be defined for an atlas of $\Hoeld{2}$ regularity. Actually \cref{thm:comparison} can be proved more general statement on a class $\mathcal{M}^\Dim(\Hoeld{m}[\alpha]\leq_r^{\textnormal{harm}} Q)$ of Riemannian manifolds with harmonic $\Hoeld{}[\alpha]$-bounded metric tensor (see \cref{ssec:functionNorms}). Such metrics have been intensively studied \cites{Taylor06}{LytchakYaman06}{SamannSteinbauer18}{KunzingerSteinbauerStojkovic13}[11]{Petersen16}{JulinLiimatainenSalo17}{KordassL20} and arise for instance as Gromov-Hausdorff limits of smooth Riemannian metrics which have been studied as well to great extent \cites{Taylor07}{CheegerColding97}{Sormani12}.
These limits constitute important examples of spaces satisfying a generalized curvature condition \cite[Ch.~29-30]{Villani08}. 
The following result continues this analytical industry.
It is also potentially beneficial from a practical point of view since given a manifold one does not have to bother about differentiability of transition maps at all as will be explained in the discussion after \cref{eq:transition_estimate}.

\begin{restatable}{theorem}{comparisonHold}\label{thm:comparison_Hold}
	Let $ \Pi $ be an invariant polynomial on $ \operatorname{M}_\Dim(\C) $ ($ \operatorname{M}_\Dim(\R) $ resp.) for some even dimension $\Dim$ and $ Q,r> 0 $.
	There is a constant $ C = C(\Pi, r, Q) $ such that
	\[ 
		\left|\Pi[M]\right| \leq C \vol (M) 
	\]
	for any closed oriented $d$-dimensional $ M \in \mathcal{M}^\Dim(\Hoeld{}[\alpha]\leq_r^{\textnormal{harm}} Q) $.
\end{restatable}

\section{H\"older norm and rough atlases}
\label{sec:HolderNorm}

\subsection{H\"older norms on function spaces}
\label{ssec:functionNorms}
We will use H\"older spaces of functions $ f\colon \Omega^\Dim \to \R^\DimVar $ on an open domain $ \Omega \subset \R^\Dim $ (we will normally work in the case that $\Omega$ is bounded and has a Lipschitz boundary).
For $m = 0, 1, \ldots$ define
\begin{equation}\label{eq:nabla}
    \nabla^m f\colon \Omega \to \R^{N\cdot \Dim^m} 
\end{equation}
to be the function of all derivatives of order $m$.
\begin{subequations}
Further let
\begin{equation}\label{eq:maxNorm}
        |(x_1, \ldots, x_\Dim)|
    \coloneqq
        \max\{|x_1|, \ldots, |x_\Dim|\}
\end{equation}
for $ x = (x_1, \ldots, x_\Dim) \in \R^\Dim$ (which is in contrast to the Euclidean norm $  \dist{0}{x} $).
Recall that the H\"older norm, for $\alpha \in [0,1]$, $m = 0, 1, \ldots$, and a domain $\Omega\subset\R^n$, is given by
\begin{align}
\label{eq:Holder_norm}
        \|f\|_{\Hoeld{m}[\alpha]}
    &\coloneqq
        \sum_{k=0}^m \|\nabla^k f\|_{\Hoeld{0}} + \|\nabla^k f\|_\alpha 
\intertext{where 
$ \|f\|_{\Hoeld{0}} \coloneqq \sup\limits_{ \substack{ x, y \in \Omega \\ x \neq y } } |f(x) - f(y)|  $ and 
$\|f\|_\alpha  = 0 $ in case $\alpha = 0 $ and otherwise }
\label{eq:Holder_seminorm}
        \|f\|_\alpha 
    &\coloneqq
        \sup_{ \substack{ x, y \in \Omega \\ x \neq y } } 
            \frac{| f(x) - f(y) |}{|x-y|^\alpha}
\text{.}
\intertext{Moreover extend these norms to all functions $\Omega \to \R^n $ by assigning}
\label{eq:seminorm_infty}
		\|\nabla^k f\|_{\Hoeld{0}}, \|f\|_\alpha 
	&\coloneqq
	    \infty
\end{align}
\end{subequations}
if $f$ does not have $k$-th derivatives.
Set $ \Hoeld{m} \coloneqq \Hoeld{m}[0] $.
Denote by $
    \Hoeld{m}[\alpha](\bar{\Omega}) = \Hoeld{m}[\alpha](\bar{\Omega}, \R)
$ the normed vector spaces that comprises the real-valued functions $\Omega \to \R^N $ whose derivatives of order not greater than $m$ have a continuation to the closure $ \bar{\Omega} $ and is endowed with the norm $ \|\blank\|_{\Hoeld{m}[\alpha]} $. Note that the continuation conditions become empty as soon as $ \alpha > 0 $.
From $\|\blank\|_\alpha \leq \|\blank\|_1 $ and the mean value theorem we get the elementary estimate
\begin{equation}\label{eq:Holder_elementaryEstimate}
        \|f\|_{\Hoeld{m}[\alpha]}
    \leq
        \|f\|_{\Hoeld{m+1}}.    
\end{equation}


\begin{subequations}

We formulate a number of convenient estimates for the H\"older norm.
Let $\Omega \subset \R^\Dim $, $\mathcal{O} \subset \R^N $ be bounded open domains with Lipschitz boundary, $ \alpha \in[0,1] $, $ m \in \{0,1,\ldots\}$, and $ c>0 $ .
The following estimates hold \cite[16.5-16.7]{Dacorogna11}: 
For $ f,g \in \Hoeld{m}[\alpha](\Omega, \R) $ we have the product estimate
\begin{equation}\label{eq:prodEstimate}
		\|fg\|_{\Hoeld{m}[\alpha]}
	\leq
		C	\|f\|_{\Hoeld{m}[\alpha]} \|g\|_{\Hoeld{m}[\alpha]} 
\end{equation}
with $ C = C(\Omega, m) $.
For a matrix valued function $ A\in \Hoeld{m}[\alpha](\Omega, \R^{\Dim\cdot \Dim}) $ such that
	\[ 
		\left\|\frac{1}{\det A}\right\|_{\Hoeld{0}} \leq c
		\quad\text{and}\quad
		\|A\|_{\Hoeld{0}} \leq c
	\]
we have the bound on the reciprocal
\begin{equation}\label{eq:reciprocalEstimate}
		\|A^{-1}\|_{\Hoeld{m}[\alpha]}
	\leq
		C \|A\|_{\Hoeld{m}[\alpha]}
\end{equation}
for a constant $ C = C(c, \Omega, m) > 0 $.

For $ f \in \Hoeld{m}[\alpha](\Omega, \R^\Dim)$ and $ g \in \Hoeld{m}[\alpha](\mathcal{O}, \R) $ with $ f(\Omega) \subset \mathcal{O}$ we have the composition estimate
\begin{equation}\label{eq:compEstimate}
		\|g\circ f\|_{\Hoeld{m}[\alpha]}
	\leq
		C\left(
				\|g\|_{\Hoeld{m}[\alpha]} \|f\|_{\Hoeld{m}[\alpha]}
			+
				\|g\|_{\Hoeld{0}}
		\right)
\end{equation}
for a constant $ C = C(m, \Omega, \mathcal{O}) $.
Moreover, if $\alpha > 0$ and $m\geq 1$, we have for any two functions $u,v \in \Hoeld{m}(\Omega, \R^d) $ 
with $ c \geq \|u\|_{\Hoeld{1}} + \|v\|_{\Hoeld{1}} $
\begin{multline}\label{eq:diffCompEstimate}
		\| g \circ u - g\circ v \|_{\Hoeld{m}}
	\leq
		C 
		\|g\|_{\Hoeld{m}[\alpha]}
		\left( 1 + \|u\|_{\Hoeld{m}[\alpha]} + \|v\|_{\Hoeld{m}[\alpha]} \right)
		 \\
		\cdot\left(\|u-v\|_{\Hoeld{0}}^\alpha + \|u-v\|_{\Hoeld{m}[\alpha]}\right).
\end{multline}
where $ C = C(c, m,\Omega, \mathcal{O} ) $.
Finally, a right inverse $ f \in \Hoeld{m}[\alpha](\Omega, \R^\Dim )$ of a function $ g \in \Hoeld{m}[\alpha](\mathcal{O}, \R^d ) $, that is to say $ f(\Omega) \subset \mathcal{O}$ and $ g\circ f = \mathit{id} $,
is bounded by
\begin{equation}\label{eq:inverseEstimate}
	\|f\|_{\Hoeld{m}[\alpha]} \leq C \|g\|_{\Hoeld{m}[\alpha]}
\end{equation}
for a constant $ 
	C = C(m,\Omega,\mathcal{O}, \|g\|_{\Hoeld{1}}, \|f\|_{\Hoeld{1}} ) 
$.
\end{subequations}

\subsection{Chart norms}
		\label{ssec:chartNorms}

H\"older classes of Riemannian metrics allow to formulate celebrated regularity results in a concise fashion.
Let $\Openball{x}{r}$ denote the open ball of radius $r$ around $x$ in the metric space to which $x$ belongs, e.g.\ $\Openball{x}{r}$ denotes the Euclidean ball around the origin $0$ in the Euclidean space $\R^\Dim$.
In contrast, we will denote a closed ball by $ \Closedball{x}{r} $.
Let  $(M, g,\basePt)$ be a pointed $n$-dimensional Riemannian manifold.
We introduce norms on charts, given by pointed maps
\begin{equation}
	\chart \colon (\Openball{0}{r}, 0) \to (M,\basePt),
\end{equation}
i.e.\ $\chart(0) = \basePt$.
We will mainly use and adapt definitions from \cite[11.3.1-11.3.5]{Petersen16}.

\begin{definition}\label{def:Holder_chartNorm}
    For a chart $\chart \colon (\Openball{0}{r}, 0) \to (M, g,\basePt)$ of $M$ we define $\|\chart\|_{\Hoeld{m}[\alpha], r}^{\textnormal{harm}}$,
    the \definiendum{harmonic chart norm of $\chart $ on the scale of $r$}
    as the minimal (note that we will state only finitely many conditions) quantity $Q \geq 0$ for which the following conditions 
    are fulfilled
    \begin{subequations}
    \begin{enumerate}
        \item
            for the differentials we have the bounds $ |D\chart| \leq e^Q $ on $\Openball{0}{r}$ and $ |D\chart^{-1}| \leq e^Q $ on $ \chart(\Openball{0}{r}) $.
            Equivalently, this condition can be expressed in coordinates on $\chart$ by
            \begin{equation}\label{eq:normBd}
                e^{-2Q} \delta_{kl} v^k v^l \leq g_{kl}  v^k v^l \leq e^{2Q} \delta_{kl} v^k v^l
            \end{equation}
            for every vector $v \in \R^\Dim$.
        \item
            regarding the semi-norm from \cref{eq:Holder_seminorm,eq:seminorm_infty} we have
            \begin{equation}
                    \label{eq:normBd_Holder}
                r^{k+\alpha} \|\nabla^k g_{\dbBlank}\|_\alpha  \leq Q
            \end{equation}
            for any $ k = 0, 1, \ldots, m$, where $ g_{\dbBlank} \coloneqq \chart^* g $.
        \item
            the chart $\chart$ is harmonic, meaning that each coordinate function $x_1, \ldots, x_\Dim$ is harmonic with respect to $g_\dbBlank$, i.e.\ the Laplace-Beltrami operator vanishes
            \begin{equation}
                    \label{eq:normBd_harmonic}
                \laplacian_{g_\dbBlank} x_1 = \ldots = \laplacian_{g_\dbBlank} x_\Dim = 0
             \text{,}
            \end{equation}
            where the Laplace operator $
            	\laplacian_{g_\dbBlank}\colon f \mapsto (\sqrt{\det g_\dbBlank} \cdot g^{kl} \cdot \partial_l f) \partial_k
            $ is understood as valued in distributions and  $ g^\dbBlank $ is the inverse of $g_\dbBlank$,
            for details cf.\ \cite[3.9]{Taylor00}.
    \end{enumerate}
    \end{subequations}
\end{definition}

\begin{subequations}
We can directly extend this definition by
\begin{align*}
        \|(M,g,\basePt)\|_{\Hoeld{m}[\alpha], r}^{\textnormal{harm}}
    &=
        \inf_{\chart\colon (\Openball{0}{r},0) \to (M,\basePt) } \|\chart\|_{\Hoeld{m}[\alpha]}^{\textnormal{harm}},
\\
        \|(M,g)\|_{\Hoeld{m}[\alpha], r}^{\textnormal{harm}}
    &=
        \sup_{\basePt \in M } \|(M, g,\basePt)\|_{\Hoeld{m}[\alpha], r}^{\textnormal{harm}}
\text{.}
\end{align*}
\end{subequations}
Note that $
    \|(M,g,\basePt)\|_{\Hoeld{m}[\alpha], r}^{\textnormal{harm}}
$ is realized by charts through an application of the Arzel\`a–Ascoli theorem.
Finally, let 
\begin{equation*}
    \mathcal{M}^\Dim(\Hoeld{m}[\alpha]\leq_r^{\textnormal{harm}} Q)
\end{equation*}
denote the space of isomorphism class of $n$-dimensional pointed Riemannian manifolds $ (M, g,\basePt) $ with $\|(M, g)\|_{\Hoeld{m}[\alpha], r}^{\textnormal{harm}} \leq Q$.
This space is endowed with the Gromov-Hausdorff topology.
Note the elementary estimate \cite[\citeProposition 11.3.2 (4)]{Petersen16}
\begin{equation}
        \label{eq:distanceComparison}
        \cref{eq:normBd}
    \implies
            e^{-Q}\min\{ \dist{x}{y}, 2r - \dist{0}{x} \}
        \leq
            \dist{\chart(x)}{\chart(y)}_g
        \leq
            e^Q \dist{x}{y}
\end{equation}
for all $ x, y \in \Openball{0}{r}$ and $|\dbBlank|$ the Euclidean norm.

Having introduced a space with a global bound on the metric tensor in local coordinates, one may feel inclined to ask why we did not assume any regularity assumption on changes of coordinates?
The answer is found in Schauder estimates, standard estimates on the regularity of solutions of elliptic PDEs.
The crucial fact can be stated as follows \cite[Problem~6.1~(a)]{GilbargTrudinger15}:
On a bounded open set $\Omega$
let $u\colon \Omega \to \R $ be a $ \Hoeld{m+2}[\alpha] $-solution ($ m \geq 0 $) of $ 
    (a^{ij}(x)\partial_i\partial_j + b^i(x) \partial_i + c(x) )u = f 
$ (summation convention) and assume that the coefficients of $L$ satisfy $ a^{ij}\xi_i\xi_j \geq \lambda |\xi|^2 $ and $ 
        \| \nabla^m a \|_\alpha, \| \nabla^m b \|_\alpha, \| \nabla^m c \|_\alpha
    \leq
        \Lambda
$. 
If $\Omega' \subset \Omega$ with $\overline{\Omega'} \subsetneqq \Omega$, then
\[
        \| u\|_{\Hoeld{m+2}[\alpha]}
    \leq
        C( \|u\|_{\Hoeld{0}} + \| f \|_{\Hoeld{m}[\alpha]} )
\]
on $\Omega'$ with $ C = C(\Dim, m, \alpha, \lambda, \Lambda, \dist{\Omega'}{\setBd \Omega}_{\mathrm{H}} ) $
where $ \Dist[H](\Omega', \setBd \Omega) $ denotes the Hausdorff distance.

If we apply this statement to a transition function $\chart_i^{-1}\circ \chart_j$ for two charts $\chart_i$, $\chart_j $ with $ 
        \|\chart_i\|_{\Hoeld{m}[\alpha], r}^{\textnormal{harm}},
        \|\chart_j\|_{\Hoeld{m}[\alpha], r}^{\textnormal{harm}}
    \leq
        Q
$ we first notice that by the harmonicity condition $ 
        \| \laplacian_g \chart_i^{-1}\circ \chart_j \|_{\Hoeld{m}[\alpha]}
    =
        \| 0 \|_{\Hoeld{m}[\alpha]}
    =
        0
$ and moreover, in harmonic coordinates, as calculated e.g.\ in \cite[11.2.3]{Petersen16}, $
    \laplacian_g = g^{kl} \partial_k \partial_l
$. If we restrict to the domain $ 
        \Omega'
    \coloneqq
        \chart_j^{-1}(\chart_i(\Openball{0}{\nicefrac{r}{2}}) \cap \chart_j(\Openball{0}{\nicefrac{r}{2}}) )
$ the above estimate becomes
\begin{equation}
        \label{eq:transition_estimate}
        \| \chart_i^{-1}\circ \chart_j\|_{\Hoeld{m+2}[\alpha]}
    \leq
        C \|\chart_i^{-1}\circ \chart_j\|_{\Hoeld{0}}
    \leq
        C \cdot \nicefrac{r}{2}
\end{equation}
where $ C = C( n, m, \alpha, Q, r) $, assuming that $\chart_i^{-1}\circ \chart_j$ is already in $\Hoeld{m+2}[\alpha]$.
But actually the regularity assumption on $\chart_i^{-1}\circ \chart_j$ can be removed: 
\textcite{Taylor06} proves that $\chart_i^{-1}\circ \chart_j$ is in $\Hoeld{1}$ and also the Sobolev class $ \Sobolev{2}[p] $ for all $p\in [1,\infty)$.
By standard existence and uniqueness results \cite[\citeTheorems 6.8, 8.9]{GilbargTrudinger15} we obtain that the transition function is actually in $\Hoeld{2}[\alpha]$.
For non-harmonic chart norms a similar result with one lower degree of regularity holds \cite{Taylor06}.
This answers the initial question, i.e.\ any harmonic $\Hoeld{}[\alpha]$-regular atlas (that is a cover of $M$ by harmonic charts in which the metric tensor is $\Hoeld{}[\alpha]$-regular) is a $\Hoeld{2}$-regular atlas. But a classical theorem of Whitney \cite[Theorem~2.9]{Hirsch97} implies that this atlas (as a $\Hoeld{1}$-atlas) is compatible with exactly one smooth atlas on $M$.
Especially, every $
	(M, g) \in \mathcal{M}^\Dim(\Hoeld{m}[\alpha]\leq_r^{\textnormal{harm}} Q) 
$ is a smooth manifold with a smooth structure distinguished by $g$.

\begin{subequations}
The two basic motivations for the classes $\mathcal{M}^\Dim(\Hoeld{m}[\alpha]\leq_r^{\textnormal{harm}}Q)$ are, first, that such a class viewed as a space with the Gromov-Hausdorff topology is compact:
For every $ n \geq 2$ and $s, Q > 0 $ 
\begin{equation}\label{eq:fundamentalThm}
    \mathcal{M}^\Dim(\Hoeld{m}[\alpha]\leq_r^{\textnormal{harm}}Q)
    \text{ is compact}
\end{equation}
with respect to the Gromov-Hausdorff topology.
This statement is sometimes called Fundamental Theorem of Convergence Theory \cite[11.3.5]{Petersen16}.

Second, that such regularity is implied by geometric conditions
\cites[11.4.4]{Petersen16}{AndersonCheeger92}: Let $\iota > 0$, $\alpha \in (0,1)$, and $ \underline{\kappa} \in \R $.
For all $Q>0$ there is $ r> 0 $ such that every pointed Riemannian manifold $ (M, g,\basePt) $ satisfies
\begin{equation}\label{eq:lowerRicciBd}
    \injRad(M) \geq \iota \text{ and } \Ricci \geq \underline{\kappa}
        \implies
            (M, g,\basePt) \in
            \mathcal{M}^\Dim(\Hoeld{\alpha}\leq_r^{\textnormal{harm}} Q) 
\text{.}
\end{equation}
\Textcite{Anderson90} concluded $\Hoeld{1}[\alpha]$-regularity
\begin{equation}\label{eq:absoluteRicciBd}
    	\injRad(M) \geq \iota \text{ and } \Ricci M \subset [\underline{\kappa}, \overline{\kappa}] \leq \kappa
    \implies
		(M, g,\basePt) \in \mathcal{M}^n(\Hoeld{1}[\alpha]\leq_r^{\textnormal{harm}} Q)
\end{equation}
for any $\underline{\kappa}, \overline{\kappa} \in \R $.
\end{subequations}

As a first application we prove the following volume comparison theorem for Betti numbers:

\comparisonBetti*


\begin{proof}
	The proof relies on a corollary of \textcite[Corollary~4.3]{Bowen15}.
	There is a flaw in the proof of theorem \cite[Theorem~4.1]{Bowen15} on which this corollary is based on, but a remedied version is found in a preprint by \textcite[Theorem~2.5]{AbertBergeronBiringerGelander19}.
	
	Fix some $i = 0, \ldots, d$.
	\textcite[Corollary~4.3]{Bowen15} (note that the different assumption on the injectivity radius therein is a misprint as becomes apparent from a glance at Theorem~4.1 therein) proved that for a sequence $M_\xN$ with $\injRad(M_\xN) > \iota $, $ \Ricci(M_\xN) \geq \underline{\kappa} $, and $ \Sec(M_\xN) \leq \overline{\kappa} $ the quotients
	\[
		\frac{b_i(M_\xN)}{\vol(M_\xN)}
	\]
	converge provided that the sequence $ M_\xN $ converges with respect to the so-called Benjamini-Schramm topology.
	In this topology each $M_\xN$ is associated to a probability measure on a space $\mathbb{M}$ of equivalence classes of pointed metric measure spaces that is supported on the set $ \setBuilder{(M_\xN, p)}{ p\in M } $.
	Due to \cref{eq:lowerRicciBd,eq:fundamentalThm} this measure has support in the compact set $\mathcal{M}^\Dim(\Hoeld{}[\alpha]\leq_r^{\textnormal{harm}}Q)$ for some suitable $r$ and $Q$.
	The Benjamini-Schramm topology is defined as the weak topology of probability measures on $\mathbb{M}$.
	
	Bowen's theorem implies that the real-valued map $ \phi_i\colon M \mapsto \frac{b_i(M)}{\vol(M)} $ extends to a continuous map on the closure
	\[
			\mathcal{M}_{\iota, \underline{\kappa}, \overline{\kappa}}
		\coloneqq
			\overline{\setBuilder{M}{
				\injRad(M) > \iota, 
				\Ricci(M) \geq \underline{\kappa}, 
				\Sec(M) \leq \overline{\kappa} 
			}} \subset \Pr(\mathbb{M})
	\]
	in the space of probability measures.
	By \cref{eq:fundamentalThm} the space $\mathcal{M}^\Dim(\Hoeld{}[\alpha] \leq_r^{\textnormal{harm}}Q )$ is compact. 
	Thus it is closed in $\mathbb{M}$.
	Hence every limit probability measure in the closure $\mathcal{M}_{\iota, \underline{\kappa}, \overline{\kappa}}$ is again supported on $\mathcal{M}^\Dim(\Hoeld{}[\alpha]\leq_r^{\textnormal{harm}}Q)$.
	Since $\mathcal{M}^\Dim(\Hoeld{}[\alpha]\leq_r^{\textnormal{harm}}Q)$ is compact \cref{eq:fundamentalThm}, so is the space of probability measures thereon with respect to the weak topology \cite[11.5.4]{Dudley02}.
	Thus being a closed subset of a compact space $\mathcal{M}^\Dim(\Hoeld{}[\alpha]\leq_r^{\textnormal{harm}}Q)$ is compact.
	Hence the image $\phi_i(\mathcal{M}_{\iota, \underline{\kappa}, \overline{\kappa}})$ is compact.
	Thus this image is bounded, i.e.\ $ \left|\frac{b_i(M)}{\vol(M)}\right| $ is bounded.
	But this is the claim.
\end{proof}

\section{Euler Characteristic}
\label{sec:EulerChar}

The Euler characteristic can be expressed by integration over the Euler class $e(M) = \mathrm{Pf}(\frac{1}{2\pi}  \Riem_\nabla)$ for a metric connection $\nabla$ where $ \mathrm{Pf} $ is the Pfaffian. Since the Pfaffian is only $ \linGr{SO}{d} $-invariant, the theory does not work the same way as explained in \cref{ssec:charNumb} but works if $\nabla$ is a metric connection.
Hence the method of a piecewise Euclidean connection as used below does not work and we have to work directly with the Riemannian metric.
Luckily, under an additional upper curvature assumption \textcite{AndersonCheeger92} also provide a degree two Sobolev regularity
\begin{equation}\label{eq:absoluteRicciBd_Sobolev}
    	\injRad(M) \geq \iota \text{ and } \Ricci M \subset [\underline{\kappa}, \overline{\kappa}] \leq \kappa
    \implies
		(M, g,\basePt) \in \mathcal{M}^\Dim(\Sobolev{2}[p]\leq_r^{\textnormal{harm}} Q )
\text{,}
\end{equation}
for all $\underline{\kappa}, \overline{\kappa} \in \R $, $p \in [\Dim,\infty) $, $Q>0$ and some $r = r(\underline{\kappa}, \overline{\kappa}, p, Q)$,
where the class $\mathcal{M}^\Dim(\Sobolev{2}[p]\leq_r^{\textnormal{harm}})$ is defined in parallel to Definition~\ref{def:Holder_chartNorm} with conditions \cref{eq:normBd,eq:normBd_harmonic} retained and condition \cref{eq:normBd_Holder} replaces by
\begin{equation}
    \label{eq:normBd_Sobolev}
        r^{2 - \Dim/p} \|\nabla^k g_\dbBlank \|_{\Lebesgue{p}}
    \leq
        Q
\text{.}
\end{equation}

We turn to the goal of this section.

\comparisonEuler*

In this theorem \cref{claim:thm:comparison_Euler_geom} directly follows from \cref{eq:absoluteRicciBd_Sobolev} and \cref{claim:thm:comparison_Euler_Sobolev}.
First we need a lemma that later will also be used in the proof of \cref{thm:comparison}.

\begin{lemma}\label{lem:volBound}
	Let $ d \geq 1 $, $ r>0$, $Q>0$, $ \varrho \in (0, e^{-Q} r ] $, and $ \alpha\in(0,1] $. 
	For any $ M $ with $ \|M\|_{\Hoeld{\alpha},r} \leq Q $ and $ x\in M $ we have
	\begin{equation}\label{eq:volBound}
		\varvol(\Closedball{x}{\varrho}) \geq v
	\end{equation}
	for a constant $ v = v(d,r, Q, \varrho)>0 $.
\end{lemma}

\begin{proof}
	Choose a chart $ \chart \colon (\Openball{0}{r}, 0) \to (M,\basePt)  $ with $ \|\chart\|_{\Hoeld{}[\alpha],r} \leq Q $.
	We have
	\begin{align*}
				\vol_M \Closedball{x}{\varrho} 
		&\geq 	\vol_M \chart (\Closedball{0}{e^{-Q}\varrho}) 
	&&\text{by \cref{eq:distanceComparison}} \\
		&=    	\int_{\Closedball{0}{e^{-Q}\varrho}} \sqrt{|\det( g_{\dbBlank})|} \diff x \\
		&\geq 	\int_{\Closedball{0}{e^{-Q}\varrho}} \sqrt{|\lambda_1|^d} \diff x  
	\shortintertext{where $\lambda_1$ is the smallest eigenvalue of $g_{\dbBlank}$ at $ x $}
		&\geq \int_{\Closedball{0}{e^{-Q}\varrho}} \sqrt{(e^{-Q})^d} \diff x
	&&\text{by \cref{eq:normBd}}\\
		&= 		\vol_\text{Eucl.}(\Closedball{0}{e^{-Q}\varrho}) \cdot e^{-dQ / 2} 
	&&\eqqcolon v(d,r, Q, \varrho).
			\qedhere
	\end{align*}
\end{proof}

\begin{proof}[\cref{thm:comparison_Euler}]
	As mentioned above, we only have to prove \cref{claim:thm:comparison_Euler_Sobolev}.
	Let $M \in \mathcal{M}^\Dim(\Sobolev{2}[p]\leq_r^{\textnormal{harm}} Q) $.
	Sobolev's inequality implies $M \in \mathcal{M}^\Dim(\Hoeld{1}[\alpha]\leq_r^{\textnormal{harm}} Q) $
	for $ 2 - \nicefrac{\Dim}{p} = 1 + \alpha $.
	Thus also $M \in \mathcal{M}^\Dim(\Hoeld{}[\alpha]\leq_r^{\textnormal{harm}} Q) $.
	
	Choose a maximal collection of points $\{p_i\}_{i\in I}$ that is $ 2e^{-Q-2} r $-separated (i.e.\ any two points in the collection have distance at least $ 2e^{-Q-2} r $).
	Such a system is automatically $ e^{-Q-1} r $-covering (meaning that $
		\bigcup_{i\in I} \Openball{p_i}{e^{-Q-1} r} = M 
	$).
	Choose charts $ 
		\chart_i \colon (\Openball{0}{r}, 0) \to (M, p_i) 
	$ with $ \| \varphi_i \|_{\Hoeld{}[\alpha], r}^{\textnormal{harm}} \leq  Q $ for each $i \in I$.
	By the length comparison \cref{eq:distanceComparison} the space $M$ is covered by the family $ 
		\{\chart_i (\Openball{0}{e^{-1}r})\}_{i\in I} 
	$. By the same inequality and the fact that $\{p_i\}_{i\in I}$ is $ e^{-Q-2} r $-separated, the balls $
		\{\chart_i (\Openball{0}{ e^{-2Q - 2} r/2})\}_{i\in I} 
	$ are pairwise disjoint. This gives the estimate
	\begin{equation}\label{eq:numbBound_Euler}
		\# I \leq \frac{\vol (M)}{ v }
	\end{equation}
	for the constant $v = v(d,r, Q, e^{-2Q - 2} r/2)$ from \cref{eq:volBound}.
	
	We estimate for the Levi-Civita connection $\nabla$
	\begin{equation*}
			\left|\chi\right|
		=	\left|\int_M e(M) \right|
		=	\left|\int_M \mathrm{Pf}\left(\frac{1}{2\pi} \Riem_{\varphi_i^*\nabla} \right) \right|
		\leq\sum_{i\in I} \left|\int_{\Openball{0}{r}} 
			\mathrm{Pf}\left(\frac{1}{2\pi} \Riem_{\varphi_i^*\nabla}\right) \right|
	\text{.}
	\end{equation*}
	Remember that
	\begin{align*}
			R^k_{\lambda\mu\nu}	
		&=  
				\left(\Gamma^k_{\nu\lambda, \mu}\right)_{[\mu\nu]} 
			+	\sum_{\kappa}\left(\Gamma^k_{\mu\kappa}\Gamma^\kappa_{\nu\lambda}\right)_{[\mu\nu]}  
	\quad\text{and}
	\\
			\Gamma _{\mu\nu}^{k}
		&=	
			\frac 1 2 \sum_\kappa g^{k\kappa }
			\left(g_{\kappa \mu, \nu} + g_{\kappa \nu, \mu}-\partial _{\kappa }g_{\mu\nu}\right)
	\text{.}
	\end{align*}
	(where $ f_{,\mu} = \partial_{\mu} f = \frac{\partial}{\partial x_\mu} f $, $ h_{[\mu\nu]} \coloneqq h_{\mu\nu} - h_{\nu\mu} $, and all indices range over $\{1,\ldots,\Dim\}$ if not specified differently) and that $g^{\dbBlank}$ denotes the inverse of $g_{\dbBlank}$.
	Thus in each chart the curvature is a polynomial
	\[
		P \in 
		\R\left[g_{kl,\mu\nu}, g_{kl,\mu}, g^{kl}{}_{,\mu},  g^{kl} \mid k, l, \mu, \nu = 1, \ldots, \Dim \right] 
	\]
	(i.e.\ a polynomial in $ d^4 + d^3 +d^3 + d^2 $ variables)
	such that $
		P(g_{kl,\mu}, g^{kl}{}_{,\mu},  g^{kl}) \in \R[g_{kl,\mu\nu} \mid k, l, \mu, \nu = 1, \ldots, \Dim ]
	$ is linear.
	By Sobolev's inequality the variable $ g_{kl,\mu} $ are $ \Hoeld{}[\alpha] $-bounded, so certainly also $ \Hoeld{0} $-bounded.
	By \cref{eq:reciprocalEstimate} also the $g^{kl}$'s are $ \Hoeld{1}[\alpha] $-bounded.
	Thus the $g^{kl}{}_{,\mu}$'s and  $g^{kl}$'s are $\Hoeld{0} $-bounded.
	
	Therefore the integrand $\mathrm{Pf}\left(\frac{1}{2\pi} \Riem_{\varphi_i^*\nabla}\right)$ is a polynomial of degree $\nicefrac{\Dim}{2}$ in the $g_{kl,\mu\nu}$'s.
	Each $ g_{kl,\mu\nu} $ is $\Lebesgue{p}$-bounded with $p > \nicefrac{\Dim}{2}$.
	Due to Hölder's inequality each monomial $ \prod_{n=1}^{d/2} g_{k_n l_n,\mu_n\nu_n} $ is $ \Lebesgue{1} $-bounded by $C(r)\cdot \prod_{n=1}^{d/2} \left\|g_{k_n l_n,\mu_n\nu_n}\right\|_{\Lebesgue{p}} $ for some constant $C(r)$.
	Thus each integral $\int_{\Openball{0}{r}} 
				\mathrm{Pf}\left(\frac{1}{2\pi} \Riem_{\varphi_i^*\nabla}\right)$ is bounded by some constant $C(r, Q, \Dim, p)$.
	This finishes the proof together with \cref{eq:numbBound_Euler} as we have now $ 
			 	\left|\chi\right|
		\leq 	(\# I) \cdot C(r, Q, \Dim, p)
		\leq		\frac{C(r, Q, \Dim, p)}{v} \vol (M)
	 $.
\end{proof}

\section{Piecewise Euclidean connection}
\label{sec:piecewiseEuclConn}

In view of the later generalization in \cref{sec:comparison_C1alpha} we want to define the curvature tensor of a connection in a setup of a manifold $ M $ with a locally finite atlas $ 
	\{\chart_i\colon \Openball{0}{\varrho} \to M\}_{i\in I} 
$ of regularity $ \Hoeld{2}[\alpha] $. 
Such an atlas can carry at most a $ \Hoeld{1}[\alpha] $ regular tensor,
	but only a $ \Hoeld{}[\alpha] $-regular connection because the transition function $ \chart^{-1} \circ \chartVar $ from 
			a chart $ \chartVar\colon V \to M $ 
		to some 
			chart $ \chart\colon U \to M $ involves second derivatives.
This transition function is given by
\begin{align}
\nonumber
		\prescript{\chart}{}{\Gamma}_{\mu \nu}^k
	&=
		(\chart^{-1}\circ\chartVar)_*(\prescript{\chartVar}{}{\Gamma}_{\dbBlank}^{.})
	\intertext{where  
		$ x = \chart^{-1} \circ \chartVar $ and $ y = \chartVar^{-1} \circ \chart $}
	\label{eq:connectionTrafo}
	&= 
			\sum_{{\mu'}, {\nu'}, l}
				y_{{\mu'}, \mu} y_{{\nu'}, \nu}
				(\prescript{\chartVar}{}{\Gamma}_{{\mu'} {\nu'}}^l \circ y)
				(x_{k,l} \circ y)
		+ 
			\sum_{l} (x_{k,l} \circ y)	(y_{l,\mu \nu}).
\end{align}

Take a partition of unity $ \{\partUnity_i\colon M \to \R\}_{i\in I} $ compatible with the atlas,
i.e.\ $ \supp \partUnity_i \subset \chart_i( \Openball{0}{\varrho} ) $ and $ \sum_{i\in I} \partUnity_i \equiv 1 $.
With such a datum $ (\{\chart_i\}, \{\partUnity_i\}) $ we define the \definiendum{piecewise Euclidean connection} by a convex combination of the Euclidean connections on each chart
\begin{equation}\label{eq:pieceEuclConnection}
	\nabla_{ \{\chart_i\}, \{\partUnity_i\} }
		= \sum_{i\in I} \partUnity_i (\chart_i)_* \nabla_{\textnormal{Eucl.}}.
\end{equation}
This makes sense as the convex combination of connections is again a connection.
In case $ 
		\prescript{\chartVar}{}{\Gamma}_{\mu \nu}^k 
	= 	\prescript{\textnormal{Eucl.}}{}{\Gamma}_{\mu \nu}^k 
	=	0 
$ formula \cref{eq:connectionTrafo} simplifies to $
		\prescript{\chart}{}{\Gamma}_{\mu \nu}^k
	=
		\sum_{l} (x_{k,l} \circ y) \linebreak[0]	(y_{l,\mu \nu}) $.
For the connection $ \nabla_{ \{\chart_i\}, \{\partUnity_i\} } $ on $ \chart $ the Christoffel symbols are given by
 \begin{align}
 	\prescript{\chart}{}{\Gamma}_{\mu \nu}^k
 		&=  
 			\sum_{i\in I} \partUnity_i 
 				((\chart^{-1} \circ \chart_i)_{*} \prescript{\textnormal{Eucl.}}{}{\Gamma}_{\dbBlank}^{.})_{\mu \nu}^k 
 		\nonumber\\
 		&= 
 			\sum_{i\in I} \partUnity_i \sum_{l} (x^i_{k,l} \circ y^i)	(y^i_{l,\mu \nu}) 
 		\eqqcolon
 		 			\sum_{i\in I} \partUnity_i \prescript{i}{}{\Gamma}_{\mu \nu}^k 
 		\label{eq:Christoffel}
\end{align}
where $ 
		\prescript{i}{}{\Gamma}_{\mu \nu}^k 
	\coloneqq 
		((\chart^{-1} \circ \chart_i)_{*} \prescript{\textnormal{Eucl.}}{}{\Gamma}_{\dbBlank}^{.})_{\mu \nu}^k 
$, $ x^i \coloneqq \chart^{-1} \circ \chart_i $, $ y^i \coloneqq \chart_i^{-1} \circ \chart $, and, by abuse of notation, $ \partUnity_i \coloneqq \partUnity_i \circ \chart $.
 			
We (formally---in view of \cref{sec:comparison_C1alpha}) calculate the curvature tensor on $ \chart $ from the standard coordinate definition
\begin{align}
	\label{eq:RiemannianCurvature}
 	\prescript{\chart}{}{R}^k_{\lambda\mu\nu}	
 		&=  
 		 		\left(\prescript{\chart}{}{\Gamma}^k_{\nu\lambda, \mu}\right)_{[\mu\nu]} 
 		 	+ 
 		 		\sum_{\kappa}\left(
 		 			\prescript{\chart}{}{\Gamma}^k_{\mu\kappa}\prescript{\chart}{}{\Gamma}^\kappa_{\nu\lambda}
			 	\right)_{[\mu\nu]}  
 	\displaybreak[0]\\
 	\intertext{where we used the standard shorthand $ h_{[\mu\nu]} \coloneqq h_{\mu\nu} - h_{\nu\mu} $ 
 		}
 		&= 	\sum_{i\in I} 
 				 \left( (\partUnity_i \prescript{i}{}{\Gamma}_{\nu \lambda}^k)_{,\mu} \right)_{[\mu\nu]}
 			+ 
 			\sum_{\kappa}\left(
 			 		 			\prescript{\chart}{}{\Gamma}^k_{\mu\kappa}\prescript{\chart}{}{\Gamma}^\kappa_{\nu\lambda}
 						 	\right)_{[\mu\nu]} 
	 \nonumber\displaybreak[0]
	 \intertext{setting $ 
	  			A \coloneqq 
	  				\sum_{\kappa}\left(
	  		 			\prescript{\chart}{}{\Gamma}^k_{\mu\kappa}\prescript{\chart}{}{\Gamma}^\kappa_{\nu\lambda}
	 			 	\right)_{[\mu\nu]}  
	  		 $ and $
	  		 		B_i 
	  		 	\coloneqq 
	  		 		\left(\partUnity_{i, \mu} \prescript{i}{}{\Gamma}_{\nu \lambda}^k\right)_{[\mu\nu]}
	  		 $}
 		& = 
 			A +
 			\sum_{i\in I} 
 					B_i
 				+
 				\left(
	 					\partUnity_i \prescript{i}{}{\Gamma}_{\nu \lambda, \mu}^k
	 			\right)_{[\mu\nu]}
 			\nonumber\displaybreak[0]\\
 		& = 
 			A +
 			\sum_{i\in I} B_i + \partUnity_i \bigl(\textstyle
 			 				\sum_{l} 
 			 					((x^i_{k, l} \circ y^i ) y^i_{l, \nu\lambda})_{,\mu}
 			 					\bigr)_{[\mu\nu]}
 			\nonumber\displaybreak[0]\\
 		&=
	 		A + \sum_{i\in I} B_i + \partUnity_i \sum_{l} 
	 				\big((x^i_{k, l} \circ y^i )_{,\mu} y^i_{l, \nu\lambda}\big)_{[\mu\nu]}
	 				+
	 					\underbrace{\big((x^i_{k, l} \circ y^i ) y^i_{l, \nu\lambda\mu}\big)_{[\mu\nu]}}_{=0}
 		\nonumber \\
 		&= 
 			\bigg(
 			\sum_{\kappa}
	 			\prescript{\chart}{}{\Gamma}^k_{\mu\kappa}\prescript{\chart}{}{\Gamma}^\kappa_{\nu\lambda}
 			+
 			\sum_{i\in I} 
 					\partUnity_{i, \mu} \prescript{i}{}{\Gamma}_{\nu \lambda}^k
 				+	\partUnity_i \sum_{l} (x^i_{k, l} \circ y^i )_{,\mu} y^i_{l, \nu\lambda}
		 		\bigg)_{[\mu\nu]}
\text.
\label{eq:curvature_connection}
\end{align}
The trick is now to define $ \Riem_{\nabla_{ \{\chart_i\}, \{\partUnity_i\} } } $ not by \cref{eq:RiemannianCurvature},
	which involves uncontrollable third derivatives, but by formula \cref{eq:curvature_connection}.
In the generalization of \cref{sec:comparison_C1alpha} the term \cref{eq:RiemannianCurvature} will not even be defined.

\begin{lemma}\label{lem:curvature_CalphaBound}
	The coordinate function \cref{eq:curvature_connection} is $ \Hoeld{}[\alpha] $-bounded provided that the transition functions involved are $ \Hoeld{2}[\alpha] $-bounded and that the partition of unity is $ \Hoeld{1}[\alpha] $-bounded.
\end{lemma}

\begin{proof}
	This is a direct consequence of the estimates \cref{eq:prodEstimate,eq:compEstimate}.
\end{proof}

\begin{definition}\label{def:randomCurvature}
	Let $ \{\chart_i\}$ be a locally finite $\Hoeld{2}$-atlas of $ M $ and $ \{\partUnity_i\}  $ a corresponding partition of unity.
	We say that the function on a chart $ \chart $ defined by \cref{eq:curvature_connection} is the \definiendum{curvature tensor} of the piecewise Euclidean connection $ \nabla_{ \{\chart_i\}, \{\partUnity_i\} } $ as defined in \cref{eq:pieceEuclConnection} and denote it by $ 
		\Riem( \{\chart_i\}, \{\partUnity_i\} ) \coloneqq \Riem_{\nabla_{ \{\chart_i\}, \{\partUnity_i\} }} $.
\end{definition}

\section{Proof of the comparison theorem}
	\label{sec:comparisonThm}

\comparison*

\begin{proof}
	By \cref{eq:lowerRicciBd} we know that all Riemannian manifolds under consideration are all in a class $\mathcal{M}^\Dim(\Hoeld{}[\alpha]\leq_r^{\textnormal{harm}} Q)$ for any $r$ and $Q = Q(\Dim, \underline\kappa, \iota, \alpha, r)$.
	Let $M \in \mathcal{M}^\Dim(\Hoeld{}[\alpha]\leq_r^{\textnormal{harm}} Q) $.
	As in the proof of \cref{thm:comparison_Euler} we can use Lemma~\ref{lem:volBound} to find an atlas of charts $ 
			\chart_i \colon (\Openball{0}{r}, 0) \to (M, p_i) 
	$ with $ \| \varphi_i \|_{\Hoeld{}[\alpha], r}^{\textnormal{harm}} \leq  Q $ for each $i \in I$ such that
	$M$ is covered by $ \{ \varphi_i(\Openball{0}{ e^{-Q-1} r / 2}) \}_{i\in I} $ and
	\begin{equation}\label{eq:numbBound}
		\# I \leq \frac{\vol (M)}{ v }
	\end{equation}
	for the constant $v = v(d,r, Q, e^{-2Q - 2} r/4)$ from \cref{eq:volBound}.
	
	We can define a partition of unity on $M$ by choosing a smooth bump function on $ b\colon \R^d\to [0,1] $ such that $ \supp b \subset \Openball{0}{\nicefrac r 2} $ and $b|_{\Closedball{0}{e^{-1}r/2}} = 1$, and setting
	\begin{align*}
			\tilde{\partUnity}_i(x) 
		&\coloneqq
			\begin{cases*}
				b \circ \chart_i^{-1} 	\protect& if $x \in \chart_i(\Openball{0}{\nicefrac r 2})$ \\
				0						\protect& else
			\end{cases*}
	\qquad\text{and}\qquad
	\\
			\partUnity_i(x)
		&\coloneqq
			\frac{1}{\sum_{j\in I} \tilde{\partUnity}_j(x) } \tilde{\partUnity}_i (x)
	\end{align*}
	as maps $\tilde{\partUnity}_i\colon M \to [0,\infty)$ and $\partUnity_i\colon M \to [0,1]$. Note that in the denominator at least one summand $ \tilde{\partUnity}_j(x) $ is 1.
	On any chart $ \chart \in \{\chart_i\}_{i\in I} $ the function $\partUnity_i$ reads in $\chart$-coordinates $
			\partUnity_i \circ \chart (x) 
		= 
			(\sum_{j \in I} \tilde{\partUnity}_j\circ\chart )^{-1} \cdot b \circ \chart_i^{-1} \circ \chart
	$.
	Combining the estimate on the transition functions \cref{eq:transition_estimate} with \cref{eq:prodEstimate,eq:inverseEstimate,eq:compEstimate,eq:reciprocalEstimate} we obtain a uniform $\Hoeld{1}[\alpha]$-bound on the $\partUnity_i$'s.
	
	Using \cref{eq:numbBound} we estimate
	\begin{align*}
			\left|\Pi[M]\right|
		&=
			\left|
				\sum_i \int_{\Openball{0}{r/2} } \partUnity_i \circ \chart_i \cdot  \Pi(\prescript{\chart_i}{}\Riem^{\blank}_{\triBlank}) \diff x
			\right|
	\\
		&\leq
			\sum_i \int_{\Openball{0}{r/2} }  \left|\Pi(\prescript{\chart_i}{}\Riem^{\blank}_{\triBlank}) \right| \diff x
	\\
		&\leq
			\frac{\vol (M)}{ v } \cdot \sup_{i\in I} \int_{\Openball{0}{r/2} }  \left|\Pi(\prescript{\chart_i}{}\Riem^{\blank}_{\triBlank}) \right| \diff x
	\text{.}
	\end{align*}
	The integrand $ \Pi(\prescript{\chart_i}{}\Riem^{\blank}_{\triBlank}) $ is a polynomial of degree $\nicefrac{d}{2}$ in the $\prescript{\chart_i}{}\Riem^{\blank}_{\triBlank}$'s, which are $\Hoeld{}[\alpha]$-bounded by \cref{eq:transition_estimate} and Lemma~\ref{lem:curvature_CalphaBound}.
	Hence, by \cref{eq:prodEstimate}, there is a $\Hoeld{0}$-bound on the integrand.
	Thus  $\int_{ \Openball{0}{r/2} }  |\Pi(\prescript{\chart_i}{}\Riem^{\blank}_{\triBlank}) | \diff x$ is bounded by a constant.
\end{proof}

\section{Comparison theorme for $\Hoeld{1}[\alpha]$-regular manifolds}
\label{sec:comparison_C1alpha}

In this section we extend the proof of \cref{thm:comparison} to the class of Hölder regular manifolds $
	 \mathcal{M}^\Dim(\Hoeld{}[\alpha]\leq_r^{\textnormal{harm}} Q)
$, cf.\ \cref{ssec:chartNorms}:
\comparisonHold*
The proof works the same way: Just don't read the first sentence; the rest goes through with the exception that one has to pay attention that the integral in the proof actually calculates the characteristic number (Proposition~\ref{prop:connIndep}).
To check this it will be necessary as well to assure the otherwise trivial fact that the piecewise Euclidean connection is actually a connection (Lemma~\ref{lem:coordinateIndep}).
This is to say that the piecewise Euclidean connection---defined by \cref{eq:pieceEuclConnection}---is a tensor, i.e.\ is coordinate independent:
\begin{restatable}{lemma}{coordinateIndep}\label{lem:coordinateIndep}
	Let $ M $ be a manifold, $ \chart_i\colon U_i \to M $ be charts in a $ \Hoeld{2} $-atlas of $ M $, and $ \partUnity_i  $ be a corresponding $ \Hoeld{1} $-partition of unity.
	For any two charts $ \chart \colon U \to M $ and $ \chart' \colon U'\to M $ the expression defined by \cref{eq:curvature_connection} coincide on $ \chart(U) \cap \chart'(U') $ as (3,1)-tensor, i.e.\
	\[ 
 		\sum_{\lambda', \mu', \nu', k'} 
 				(\prescript{\chart^{\mathrlap{\prime}}}{}{R}^{k'}_{\lambda'\mu'\nu'}
 					\circ x') x'_{\lambda', \lambda} x'_{\mu', \mu} x'_{\nu', \nu} (x_{k, k'} \circ x') 
 	=
 		\prescript{\chart}{}\Riem^k_{\lambda\mu\nu}.
	\]
\end{restatable}
The standard textbook calculation does not work as it involves third derivatives.
But the calculation can be done by considerably more labor without third derivatives and is found in \cref{sec:coordIndep}.

We turn to the second hurdle.
Recall that in local coordinates the cohomology class $ (\phi(\Pi) )_\Dim \in H^d(M) $ is a polynomial in a curvature tensor $ \Riem_\nabla $ of some connection $\nabla$.
Let's write $ \Pi(\Riem^k_{\lambda\mu\nu}) $ for the differential form.

\begin{lemma}\label{lem:inverseConv}
		Let $ \Omega, \mathcal{O} $, and $ \mathcal{O}' $ be domains with $ \mathcal{O} \subset \mathcal{O}' $, $  m\geq 1 $ an integer, and $ \alpha\in(0,1] $.
		Let $ f_n, f \in \Hoeld{m}[\alpha]( \Omega, \R^d ) $ with inverses 
			$ g_n, g \in \Hoeld{m}[\alpha](\mathcal{O}, \R^d ) $
			(i.e.\ $ g_n\circ f_n = \mathit{id}_{\Omega} $ and $ g\circ f = \mathit{id}_{\Omega} $)
			 such that 
		\[ 
			f_n \xrightarrow{n\to \infty} f \quad \text{in $ \Hoeld{m}[\alpha] $-norm}.
		\]
		Assume further that on $ \mathcal{O}' $ the converse equalities 
			$ f_n\circ g_n = \mathit{id}_{\mathcal{O}'} $ and $ f\circ g = \mathit{id}_{\mathcal{O}'} $ hold.
		
		Then the inverses $ g_n $ converge to $ g $ in $ \Hoeld{m} $-norm on $ \mathcal{O}' $. 	
\end{lemma}

\begin{proof}
	Let $ c \geq 0 $ such that $ \|g_n\|_{\Hoeld{1}}, \|g\|_{\Hoeld{1}}, \|f_n\|_{\Hoeld{1}}, \|f\|_{\Hoeld{1}} \leq c $ for all $n$.
	By abuse of notation we write $ g_n $ and $ g $ for the restrictions $ g_n|_{\mathcal{O}'} $ and $ g|_{\mathcal{O}'} $.
	We estimate on the domain $\mathcal{O}'$
	\begin{align*}
			\| g_n - g \|_{\Hoeld{m}}
		&=
			\|g_n \circ f \circ g - g_n \circ f_n \circ g \|_{\Hoeld{m}}
	\\
	\shortintertext{apply \cref{eq:diffCompEstimate} to $ u = f \circ g $ and $ v = f_n \circ g $}
		&\leq
			\begin{multlined}[t]
			C 
				\|g\|_{\Hoeld{m}[\alpha]}
				\left( 1 + \|f \circ g\|_{\Hoeld{m}[\alpha]} + \|f_n \circ g\|_{\Hoeld{m}[\alpha]} \right)
			\\
				\cdot \big(\|(f-f_n)\circ g\|_{\Hoeld{0}}^\alpha + \|(f-f_n)\circ g\|_{\Hoeld{m}[\alpha]}\big).
			\end{multlined}
	\end{align*}
	The first factors are bounded by assumption.
	The second factor is bounded by assumptions and the composition estimate \cref{eq:compEstimate}.
	The last factor converges to 0 as $ n \to \infty $ due to \cref{eq:compEstimate}
	and the assumption that $ f_n \to f $ in $ \Hoeld{m}[\alpha] $-norm.
\end{proof}


\begin{proposition}\label{prop:connIndep}
	For a finite index set $I$ let $ \{\chart_i\colon U_i \to M\}_{i\in I}$ be a $\Hoeld{2}[\alpha]$-atlas of a closed oriented manifold $ M $ with $\alpha > 1$ and $ \{\partUnity_i\}_{i\in I}  $ a corresponding partition of unity on $ M $ such that $ \overline{\supp \partUnity_i} \subset \chart_i(U_i) $.
	Choose some invariant polynomial $ \Pi $.
	
	Then the value
	\[ 
			\int_M \Pi (\Riem( \{\chart_i\}, \{\partUnity_i\} ))[M]
		\coloneqq 
			\sum_i \int_{U_i } \partUnity_i \Pi(\prescript{\chart_i}{}\Riem^{\blank}_{\triBlank}) \diff x, 
	\]
	coincides with $ \Pi[M]$ as defined in \cref{eq:characteristicNumb}.
\end{proposition}

\begin{proof}
It is a standard result that the lemma holds for smooth connections \cite[23.5]{Tu17}, namely even the entire de Rham class of the form $ \varphi(\Pi) $ is independent of the connection.
This result extends using mollification:

Being a manifold, $ M $ is metrizable. Hence we can assume that $ M $ carries some metric $ d $.
Let $ \Dist[H] $ denote the Hausdorff distance with respect to $ d $ between closed subsets of $M$.
Set $V_i \coloneqq \chart_i(U_i)$.
For any chart $ \chart_i $ choose some $ \varepsilon_i < \Dist[H](\supp \partUnity_i, \partial V_i ) $.
For any chart $ \chart_i\colon U_i \to V_i \subset M $ we can define a mollification by choosing a finite cover $ V_{ij} $ of $ \overline{V_i} $
	corresponding to a finite set of \textit{smooth} charts $ \{\chart_{ij}\colon U_{ij} \to V_{ij} \subset M\}_{j \in J_i} $ for $ V_i $.
Moreover we choose
	corresponding \textit{smooth} functions $ \{\partUnity_{ij}\colon M\to [0,1]\}_{j \in J_i}  $ that are a partition on unity for the $\varepsilon_i$-thickening of $ \supp \partUnity_i $,
	i.e.\ for all $i$ we require
\begin{align}
\nonumber
		\supp \partUnity_{ij} 
	&\subset V_{ij}	&& \text{for all } j,
\\
\nonumber
	\bigcup\nolimits_j \supp \partUnity_{ij} &\subset V_i,
&&\text{and}  
\\
\label{eq:subPartUnity}
		\supp (\partUnity_i)^{\varepsilon_i)} 
	&\subset 
		\left(\sum\nolimits_j \partUnity_{ij}\right)^{-1}(\{1\})
\end{align}
where $
		A^{\varepsilon)}
	\coloneqq
		\setBuilder{x \in M}{ \Dist(x, A) < \varepsilon }
$ for $A \subset M$.
On each chart $ \chart_{ij} $ we mollify the compactly supported function $ 
		(\partUnity_{ij} \circ \chart_{ij}) \cdot  (\chart_i^{-1} \circ \chart_{ij} )
	\colon 
		U_{ij} \to \R^d 
$ by convolution on the chart $ \chart_{ij} $ with a function $ \phi_\delta $ which has support in $ \Closedball{0}{\delta} $ and satisfies $ \int \phi_\delta \diff x = 1 $,
see \cite[Lemma~1.2.3]{Hormander03} for existence.
The resulting function 
\begin{align*}
		\chartVar_{ij}^\delta(x) 
	&\coloneqq 
		\phi_\delta * \left(  (\partUnity_{ij} \chart_i^{-1} ) \circ \chart_{ij} \right)(x)
\\
	&=
		\int_{\R^d} 
					\phi_\delta (\xi - x) 
			\cdot 	\partUnity_{ij} \circ \chart_{ij}(\xi)
			\cdot 	\chart_i^{-1} \circ \chart_{ij}(\xi)
		\diff \xi 
\\
\intertext{has support in $ V_{ij} $ for $ \delta < \delta_{ij} \coloneqq \Dist[H](\supp \partUnity_{ij}, \setBd U_{ij}) \neq 0 $.
		Summing up we define for each $i\in I$ a function $ M \to \R^d $}
		\chartVar_{i}^\delta(x) 
	&\coloneqq
		\sum\nolimits_j \chartVar_{ij}^\delta \circ \chart_{ij}^{-1} (x)
\end{align*}
which is well-defined for $ \delta < \delta_i \coloneqq \min_j \delta_{ij} \neq 0 $.
Moreover this function has the property that for sufficiently small $ \delta $ the support is contained in $ V_i $.
Finally, we observe that by basic mollification results the functions 
	$ \chartVar_{ij}^\delta $ are smooth and $ \Hoeld{2}[\alpha] $-converge to $ \chart_i^{-1}\circ \chart_{ij} $ as $ \delta\to 0 $.
Therefore, for each $i\in I$, $ \chartVar_{i}^\delta $ converges in $ \Hoeld{2} $ to $ \chart_i^{-1} \cdot \sum\nolimits_j \partUnity_{ij} $ as $ \delta\to 0 $, meaning that for any smooth chart $\chart$ the sequence $\chartVar_{i}^\delta \circ \chart$ converges in $ \Hoeld{2}[\alpha] $ to $ \chart_i^{-1} \circ\chart \cdot \sum\nolimits_j \partUnity_{ij} \circ\chart $.
Thus, by \cref{eq:subPartUnity}, $ \chartVar_{i}^\delta $ converges in $ \Hoeld{2}[\alpha] $ to $ \chart_i^{-1} $ on the support of $\partUnity_i$ as $ \delta\to 0 $.
Hence, due to Lemma~\ref{lem:inverseConv}, for any chart $ \varphi $ the sequence $ \chart \circ (\chartVar_{i}^\delta)^{-1}$ converges in $ \Hoeld{2} $ to $ \chart_i$ on $ \bigcap_{\delta > 0} \chartVar_{i}^\delta (\supp \partUnity_i) $ as $ \delta \to 0 $.

In a similar fashion we define, for sufficiently small $ \delta $, the mollification of the partition of unity $ \{\partUnity_i\} $
\begin{align*}
		\partUnityVar_{ij}^\delta(x) 
	&\coloneqq 
		\phi_\delta * \left( (\partUnity_{ij} \partUnity_i) \circ \chart_{ij} \right) (x)
	=
		\int_{\R^d} \phi_\delta (\xi - x) \cdot \partUnity_{ij}(\chart_{ij}(\xi)) \diff \xi \\
	\tilde{\partUnityVar}_{i}^\delta(x) &\coloneqq
		\sum_{j\in I} \partUnityVar_{ij}^\delta( \chart_{ij}^{-1} (x)).
\\
\intertext{To preserve the partition of unity property we normalize with respect index $ i $, for all $i$ we set}
		\partUnityVar_{i}^\delta(x) 
	&\coloneqq
		\frac{\tilde{\partUnityVar}_{i}^\delta(x)}{\sum\nolimits_{j\in I} \tilde{\partUnityVar}_{i}^\delta(x)}
\end{align*}
and observe that $ \partUnityVar_{i}^\delta $ converges in $ \Hoeld{2} $ to $ \partUnity_i \cdot \sum\nolimits_j \partUnity_{ij} = \partUnity_i $ as $ \delta\to 0 $.
For each $ i $ we can choose $ \delta >0 $ sufficiently small so that $ \supp \partUnityVar_{i}^\delta \subset \supp (\partUnity_i)^{\varepsilon_i)} $.

Now we can define the mollified connection:
	Each map $ \chartVar_{i}^\delta $ is a diffeomorphism from $ \bigl(\sum\nolimits_j \partUnity_{ij}\bigr)^{-1}(\{1\}) $ 
		onto a subset of $ \R^d $ and hence diffeomorphism on $ \supp \partUnityVar_i^\delta $.
Thus we can define a smooth connection on $ M $ given by
\begin{equation}\label{eq:mollifiedConnection}
				\nabla^\delta 
	\coloneqq 	\nabla_{ \{\chartVar_i\}, \{\partUnityVar_i\} }
	=			\sum_{i\in I} \partUnityVar_i^\delta (\chartVar_i^\delta)^* \nabla_{\textnormal{Eucl.}}
\end{equation}
that is well-defined for $\delta < \min_i \delta_i $.
Let $ \Riem_{\nabla^\delta} $ denote the curvature tensor induced by this connection according to \cref{def:randomCurvature} and let $ \prescript{\chart}{\delta}\Riem^\blank_{\triBlank} $ denote the coordinate representation of this connection in a chart $\chart$ according to \cref{eq:curvature_connection}.
In this formula the roles of the $\chart_i$'s and $\partUnity_i$'s is now taken over by the $(\chart_i^\delta)^{-1}$'s and $\partUnityVar_i^\delta$'s obtaining
\begin{equation*}
	\prescript{\chart}{\delta}\Riem^k_{\lambda\mu\nu} =
 		\bigg(
 		\sum_{\kappa}
 		 \prescript{\chart}{\delta}{\Gamma}^k_{\mu\kappa}\prescript{\chart}{\delta}{\Gamma}^\kappa_{\nu\lambda}
 		 +
 		\sum_{i\in I} 
 				\partUnityVar^\delta_{i, \mu} \prescript{i}{\delta}{\Gamma}_{\nu \lambda}^k
 			+	\partUnityVar^\delta_i \sum_{l} 
 				(\prescript{\delta}{}{x}^i_{k, l} \circ \prescript{\delta}{}{y}^i )_{,\mu} \prescript{\delta}{}{y}^i_{l, \nu\lambda}
 		\bigg)_{[\mu\nu]}
\end{equation*}
where
\begin{align*}
		\prescript{i}{\delta}{\Gamma}_{\mu \nu}^k 
	&\coloneqq 
		\sum_{l} (\prescript{\delta}{}{x}^i_{k,l} \circ \prescript{\delta}{}{y}^i)	(\prescript{\delta}{}{y}^i_{l,\mu \nu})
,&
		\prescript{\chart}{\delta}{\Gamma}^k_{\mu\kappa}
	&\coloneqq
		\sum_{i\in I} \partUnityVar_i \prescript{i}{\delta}{\Gamma}_{\mu \nu}^k
,
\\
	\prescript{\delta}{}{x}^i &\coloneqq \chart^{-1} \circ (\chartVar^\delta_i)^{-1}	
,&
	\prescript{\delta}{}{y}^i &\coloneqq \chartVar_i^\delta \circ \chart
\text{,}
\end{align*}
and---by abuse of notation---$ \partUnityVar^\delta_i \coloneqq \partUnityVar^\delta_i \circ \chart $.
Since \cref{eq:curvature_connection} involves only second derivatives, the $\Hoeld{2}$-convergences
$ \prescript{\delta}{}{x}^i \to x^i = \chart^{-1} \circ \chart_i $ on $\chart_i(\supp \partUnity_i)$,
$ \prescript{\delta}{}{y}^i \to x^i = \chart_i^{-1} \circ \chart $ on $\chart (\supp \partUnity_i)$,
and
$\partUnityVar_{i}^\delta \to \partUnity_i $ as $\delta \to 0$ proved above
imply by \cref{eq:prodEstimate,eq:reciprocalEstimate,eq:compEstimate,eq:inverseEstimate,eq:diffCompEstimate}
\begin{equation}\label{eq:curvatureSmoothing}
	\prescript{\chart}{\delta}\Riem^{\blank}_{\triBlank}
	\xrightarrow{\delta\to 0}
	\prescript{\chart}{}\Riem^{\blank}_{\triBlank}
\qquad
	\text{in	$\Hoeld{0}$.}
\end{equation}

Now we can check the claim on characteristic numbers.
By coordinate independence, i.e. Lemma~\ref{lem:coordinateIndep}, we can calculate $\int_M \Pi (\Riem ( \{\chart_i\}, \{\partUnity_i\} ) ) $ as
\begin{align*}
		\int_M \Pi (\Riem ( \{\chart_i\}, \{\partUnity_i\} )) 
	&= 
		\adjustlimits{\sum}_{i} {\sum}_{j\in J_i} \int_{U_{ij} }
			((\partUnity_i \partUnity_{ij})  \circ \chart_{ij}) 
			\cdot 
			\Pi(\prescript{\chart_{ij}}{}{\Riem}^\blank_{\triBlank}) 
		\diff x 
\shortintertext{by \cref{eq:curvatureSmoothing}}
	&=
		\adjustlimits{\sum}_{i} {\sum}_{j\in J_i} \int_{U_{ij} }
			((\partUnity_i \partUnity_{ij})  \circ \chart_{ij}) 
			\cdot 
			\Pi\left( \lim_{\delta\to 0} \prescript{\chart_{ij}}{\delta}\Riem^\blank_{\triBlank} \right) 
		\diff x 
\shortintertext{by dominated convergence theorem}
	&= 
		\lim_{\delta\to 0} \adjustlimits{\sum}_{i} {\sum}_{j\in J_i} \int_{U_{ij} } \int_{U_{ij} }
			((\partUnity_i \partUnity_{ij})  \circ \chart_{ij})  
			\cdot 
			\Pi\left( \prescript{\chart_{ij}}{\delta}\Riem^\blank_{\triBlank}\right) 
		\diff x 
\\
	&= \lim_{\delta\to 0} \int_M \Pi(\Riem_{\nabla^\delta})
\text{.}
\end{align*}
Since $ \int_M \Pi(\Riem_{\nabla^\delta}) $ is independent of $\nabla^\delta$, this proves that $ \int_M \Pi (\Riem( \{\chart_i\}, \{\partUnity_i\} ) ) $ equals the actually characteristic number of $ M $ with respect to $ \Pi $ and is therefore independent of the choice of the atlas $ \{U_i\} $.
\end{proof}

\begin{remark}
	The proof of Proposition~\ref{prop:connIndep} also shows that the entire characteristic class has a representative of uniformly bounded $\Hoeld{2}[\alpha]$-regularity in the sense that there is a sequence of representatives of $\phi(\Pi)$ such that in some charts with harmonic chart norm on the scale of $r$ bounded by $Q$ the sequence converges in $ \Hoeld{0} $ to the rough $\Hoeld{}[\alpha]$ form defined by $\Pi$ and the locally Euclidean connection.
	By a rough $\Hoeld{}[\alpha]$ form we mean a section in the exterior algebra that is $\Hoeld{}[\alpha]$-regular with respect to the Euclidean identification in the chart.
\end{remark}

\appendix

\section{Coordinate independence}
	\label{sec:coordIndep}

We check that the curvature defined in the notation of \cref{sec:piecewiseEuclConn} by \ref{eq:curvature_connection} which can be stated as
\begin{equation}\label{eq:curvature_connection_var}
		\prescript{\chart}{}\Riem^k_{\lambda\mu\nu}
	= 
 		\sum_{i\in I} \partUnity_i \sum_{l} 
 			\left((x^i_{k, l} \circ y^i )_{,\mu} y^i_{l, \nu\lambda}\right)_{[\mu\nu]}
 		+ \left(\partUnity_{i, \mu} \prescript{i}{}{\Gamma}_{\nu \lambda}^k\right)_{[\mu\nu]}
 		+ \sum_{\kappa}
\left(\prescript{\chart}{}{\Gamma}^k_{\mu\kappa}\prescript{\chart}{}{\Gamma}^\kappa_{\nu\lambda}\right)_{[\mu\nu]}
\end{equation}
is actually coordinate independent.
The standard textbook calculation could be referred to if a third derivative of transition maps would exists.

\coordinateIndep*

\begin{proof}
	We repeat all crucial definitions in the primed and non-primed versions
	\begin{align*}
			x 
		&= 
			\chart^{-1} \circ \chart', 
	& 	
			x' 
		&= 
			\chart'^{-1} \circ \chart,  
	\\
			x^i 
		&= 
			\chart^{-1} \circ \chart_i, 
	& 	
			x'^i 
		&= 
			\chart'^{-1} \circ \chart_i, 
	\\
			y^i 	
		&= 
			\chart_i^{-1} \circ \chart,
	&		y^{\prime i} 	
		&= 
			\chart_i^{-1} \circ \chart',
	\\
			\prescript{i}{}{\Gamma}_{\mu \nu}^k
		&=
			\sum_{l} (x^i_{k,l} \circ y^i)	(y^i_{l,\mu \nu}),
	&		
			\prescript{i}{}{\Gamma}_{\mu \nu}^{\prime k}
		&=
			\sum_{l} (x^{\prime i}_{k,l} \circ y^{\prime i})	(y^{\prime i}_{l,\mu \nu}),
	\\
			\prescript{\chart}{}{\Gamma}_{\mu \nu}^k
		&=  
			\sum_{i\in I} \partUnity_i \prescript{i}{}{\Gamma}_{\mu \nu}^k, 
	&	
			\prescript{\chart'}{}{\Gamma}_{\mu \nu}^k
		&=  
			\sum_{i\in I} \partUnity_i \prescript{i}{}{\Gamma}_{\mu \nu}^{\prime k}.			
	\end{align*}
	defined on the suitable domains and the abuse of notation $\partUnity_i = \partUnity_i \circ \chart $.
	Additionally, we introduce the shorthand
	\[ 
		(h_{\mu\mu'\nu\nu'})_{[\mu\nu]'} \coloneqq h_{\mu\mu'\nu\nu'} - h_{\nu\nu'\mu\mu'}
	\text{.} 
	\]
	Observe that in case $ h_{\mu\mu'\nu\nu'} = f_{\mu\mu'} \cdot g_{\nu\nu'} $
	\begin{align}
	\nonumber
			\sum\nolimits_{\mu'\nu'} \left(h_{\mu'\mu'\nu'\nu'}\right)_{[\mu\nu]'}
		&=
				\Bigl( \sum\nolimits_{\mu'} f_{\mu\mu'} \Bigr) 
				\Bigl( \sum\nolimits_{\nu'} g_{\nu\nu'} \Bigr) 
			- 
				\Bigl( \sum\nolimits_{\nu'} f_{\nu\nu'} \Bigr) 
				\Bigl( \sum\nolimits_{\mu'} g_{\mu\mu'} \Bigr) 
	\\
	\label{eq:eliminatePrimeBracket}
		&=
			\Bigl(\sum\nolimits_{\mu'} f_{\mu\mu'} \sum\nolimits_{\nu'} g_{\nu\nu'} \Bigr)_{[\mu\nu]}.
	\end{align}
	
	We have $ x' \coloneqq \chart'^{-1} \circ \chart = x'^i \circ y^i $ and
		$ x \coloneqq \chart^{-1} \circ \chart' = x^i \circ y'^i $ where defined.
	Observe that for the differential $ D\mathit{id}\colon \R^d \to \R^d $ of the identity the following identity holds
	\begin{align}
			0 	
		&= 
			(D\mathit{id})_{,\lambda}
		= 
			(\operatorname{D} x\circ x' )_{,\lambda} 
		= 
			\left( 
				\left( \sum\nolimits_l x_{k,l} \circ x' \cdot x'_{l,\nu} \right)_{k\nu} 
			\right)_{,\lambda} 
	\nonumber
	\\
	\label{eq:eliminateDbDerivative}
		&= 
			\sum\nolimits_l \left( 
					(x_{k,l} \circ x')_{,\lambda} x'_{l,\nu} 
				+ 
					(x_{k,l} \circ x') x'_{l,\nu\lambda} 
			\right)_{k\nu}.
	\end{align}
	
	Finally, we prove the lemma. Fix any $ \lambda, \mu, \nu, k \in \{1,\ldots, \Dim\} $ and observe
\small
	\begin{align*}
			\MoveEqLeft[0]{ \sum_{\lambda', \mu', \nu', k'}
				(\prescript{\chart^{\mathrlap{\prime}}}{}{R}^{k'}_{\lambda'\mu'\nu'} 
					{}\circ x') x'_{\lambda', \lambda} x'_{\mu', \mu} x'_{\nu', \nu} (x_{k, k'} \circ x') } 
	\shortintertext{plug in definition \cref{eq:curvature_connection_var} for $ \prescript{\chart^{\mathrlap{\prime}}}{}{R}^{k'}_{\lambda'\mu'\nu'} $ and use the abuse of notation $\partUnity_{i, \mu'} = (\partUnity_i \circ \chart')_{,\mu} $}
		&= 
			\begin{multlined}[t]
				\sum_{i\in I}
				\partUnity_i
				\sum_{\lambda', k', l'}
				\sum_{\mu', \nu'}
				\left(
						\big((x^{\prime i}_{k', l'} \circ 
				 		y^{\prime i} )_{,\mu'} \circ x'\big) x'_{\mu', \mu}
				 		(y^{\prime i}_{l', \nu'\lambda'} \circ x') x'_{\lambda', \lambda}
					 	x'_{\nu', \nu} (x_{k, k'} \circ x') 
				\right)_{[\mu\nu]'}
			\\
				+ 
					\sum_{ i\in I }
					\sum_{\lambda', k'}
					\sum_{\mu', \nu'}
					\bigl(
						\partUnity_{i, \mu'} x'_{\mu', \mu}
						(\prescript{i}{}{\Gamma}_{\nu' \lambda'}' \circ x')^{k'}
						x'_{\lambda', \lambda} x'_{\nu', \nu} (x_{k, k'} \circ x')
					\bigr)_{[\mu\nu]'}
			\\  
				+ 
					\sum_{\lambda', k', \kappa'}
					\sum_{\mu', \nu'}
					\left(
						\prescript{\chart'}{}{\Gamma}^{k'}_{\mu'\kappa'}
						\prescript{\chart'}{}{\Gamma}^{\kappa'}_{\nu'\lambda'}
						x'_{\lambda', \lambda} x'_{\mu', \mu} x'_{\nu', \nu} (x_{k, k'} 
						\circ x') 
					\right)_{[\mu\nu]'}
			\end{multlined} 
	\shortintertext{apply formula \cref{eq:eliminatePrimeBracket} to each summand $\sum_{\mu', \nu'}(\ldots)_{[\mu\nu]'}$}
		&= 
			\begin{multlined}[t]
				\sum_{ i\in I } 
					\sum_{\lambda', k', l'}
					\partUnity_i 
	 				\Bigl(
		 				\textstyle
					 	\overbrace{\textstyle
					 		\sum\limits_{\mu'} ((x^{\prime i}_{k', l'} \circ 
				 		y^{\prime i} )_{,\mu'} \circ x') x'_{\mu', \mu}
				 		}^{\text{change of variables}}
					 	\cdot\sum\limits_{\nu'}
					 	\underbrace{
					 		(y^{\prime i}_{l', \nu'\lambda'} \circ x') x'_{\lambda', \lambda}
					 	}_{\text{change of variables}}
					 	x'_{\nu', \nu} (x_{k, k'} \circ x') 
				\Bigr)_{[\mu\nu]}
			\\
				+ 
					\sum_{ i\in I } \sum_{\lambda', k'}
					\Bigl(
						\underbrace{\textstyle
							\sum_{\mu'}
							\partUnity_{i, \mu'} x'_{\mu', \mu}
						}_{\text{change of variables}}
						\underbrace{\textstyle
							\sum_{\nu'}
							(\prescript{i}{}{\Gamma}_{\nu' \lambda'}' \circ x')^{k'}
							x'_{\lambda', \lambda} x'_{\nu', \nu} (x_{k, k'} \circ x')
						}_{\text{apply formula \cref{eq:connectionTrafo}}}
					\Bigr)_{[\mu\nu]}
			\\  
				+ 
					\sum_{\lambda', k', \kappa'}
					\Big(\textstyle
						\sum\limits_{\mu'}
						\prescript{\chart'}{}{\Gamma}^{k'}_{\mu'\kappa'}
						x'_{\mu', \mu}
						(x_{k, k'} \circ x')
						\sum\limits_{\nu'}
						\prescript{\chart'}{}{\Gamma}^{\kappa'}_{\nu'\lambda'}
						x'_{\lambda', \lambda}  
						x'_{\nu', \nu}  
					\Big)_{[\mu\nu]}
			\end{multlined} 
	\\
		&= 
			\begin{multlined}[t]
				\sum_{ i\in I } \sum_{k', l'}
						\partUnity_i 
					 \Bigl(\textstyle
					 	(
					 		\underbrace{x^{\prime i}_{k', l'}}_{\mathclap{=(x'\circ x^i)_{k',l'}} }  
					 		{}\circ {}
					 		\underbrace{y^{\prime i} \circ x'}_{= y^i}  
					 	)_{,\mu}
					 	\sum_{\nu'}
					 	(
					 		\underbrace{y^{\prime i}_{l', \nu'} }_{\mathclap{=(y^i \circ x)_{l',\nu'}}} 
					 		{}\circ 
					 		x'
					 	)_{, \lambda}
					 	x'_{\nu', \nu} (x_{k, k'} \circ x') 
					 \Bigr)_{[\mu\nu]} \\
			+ 	\sum_{i\in I}
				\Bigl(\textstyle
					\partUnity_{i, \mu} 
					\bigl(
						\prescript{i}{}{\Gamma}_{\nu \lambda}^k
						- \sum_{l} (x_{k,l} \circ x')	(x'_{l,\nu \lambda})
					\bigr) 
				\Bigr)_{[\mu\nu]} \\
			+ 
				\sum_{\lambda', k', \kappa', \mu', \nu', m} \Bigl(
					\prescript{\chart'}{}{\Gamma}^{k'}_{\mu'm}
					x'_{\mu', \mu}
					(x_{k, k'} \circ x') 
					\cdot
					\underbrace{ \delta_{m\kappa'} }_{ (*)} 
					\cdot
					\prescript{\chart'}{}{\Gamma}^{\kappa'}_{\nu'\lambda'}
						x'_{\lambda', \lambda}  
						x'_{\nu', \nu} 
				\Bigr)_{[\mu\nu]}
			\end{multlined} 
	\shortintertext{$(*)$: $ 
		\delta_{m\kappa'} 
		= \delta_{m\kappa'} \circ x'
		= (D \mathit{id} )_{m\kappa'}  \circ x' 
		= (D (x'\circ x))_{m\kappa'}   \circ x'
		= (\sum_{ \kappa} x'_{m, \kappa} \circ x \cdot x_{\kappa, \kappa'} ) \circ x'
		= \sum_\kappa  x'_{m, \kappa} (x_{\kappa, \kappa'} \circ x')
	$}
		&= 
			\begin{multlined}[t]
				\smashoperator{\sum_{\substack{i\in I \\ k', l' }}} 
						\partUnity_i 
					 \Biggl(\begin{multlined}[c][8cm]\textstyle
					 		\bigl(\bigl( \sum_\kappa     x'_{k', \kappa}   \circ x^i \cdot x^i_{\kappa,l'} \bigr) \circ y^i  \bigr){}_{,\mu}
					 	\\ \textstyle \cdot
					 		\sum_{\nu'}
					 		\bigl(\bigl( \sum_{l} y^i_{l', l}\circ x   \cdot x_{l, \nu'}\bigr) \circ x'\bigr){}_{, \lambda}
					 	\textstyle \cdot
					 		x'_{\nu', \nu} (x_{k, k'} \circ x') 
					 \end{multlined}
					\Biggr)_{[\mu\nu]} 
			\\
			+ 	
				\sum_{i\in I}
				\bigl(
					\partUnity_{i, \mu}
						\prescript{i}{}{\Gamma}_{\nu \lambda}^k
				\bigr)_{[\mu\nu]}
					- \sum_{l} \left(\textstyle
						\left(\sum_{i\in I} \partUnity_{i}\right)_{, \mu} 
							(x_{k,l} \circ x')	(x'_{l,\nu \lambda}) 
						\right)_{[\mu\nu]} \\
			+ 
				\sum_\kappa\Biggl(
					\underbrace{\textstyle \sum\limits_{k', \mu', m}
						\prescript{\chart'}{}{\Gamma}^{k'}_{\mu'm} 
							x'_{\mu', \mu} (x_{k, k'} \circ x') x'_{m, \kappa}
					}_{\text{apply formula \cref{eq:connectionTrafo}}}
					\underbrace{\textstyle
						\sum\limits_{\lambda', \kappa', \nu'}
						\prescript{\chart'}{}{\Gamma}^{\kappa'}_{\nu'\lambda'} 
							x'_{\lambda', \lambda}  x'_{\nu', \nu} (x_{\kappa, \kappa'} \circ x')
					}_{\text{apply formula \cref{eq:connectionTrafo}}}											 
				\Biggr)_{[\mu\nu]}
			\end{multlined} 
	\\
		&= 
			\begin{multlined}[t]
				\smashoperator{\sum_{\substack{i\in I \\ k', l' }}} 
					\partUnity_i 
					\Bigl(  \textstyle 
						\sum_{\kappa}
					 	(  x'_{k', \kappa }    \cdot x^i_{\kappa ,l'} \circ y^i  )_{,\mu}
					 	\sum_{\nu, l}
					 	( y^i_{l', l} \cdot x_{l, \nu'} \circ x')_{, \lambda}
					 	x'_{\nu', \nu} (x_{k, k'} \circ x') 
					 \Bigr)_{[\mu\nu]}
			\\
			+ 	
				\smash{\overbrace{\sum_{i\in I}
					\bigl(
						\partUnity_{i, \mu}
							\prescript{i}{}{\Gamma}_{\nu \lambda}^k\bigr)_{[\mu\nu]}}^{\eqqcolon M}}
						- \sum_{l} \Bigl(
							\smash{\overbrace{1_{, \mu} }^{=0}}
								(x_{k,l} \circ x')	x'_{l,\nu \lambda}
							\Bigr)_{[\mu\nu]} \\
			+ 
				\sum_{ \kappa}
					\Bigl(\textstyle
						\left(
							\prescript{\chart}{}{\Gamma}^k_{\mu\kappa} 
								- \sum_{l} (x_{k,l} \circ x')	x'_{l,\mu \kappa}
						\right)
						\left(
							\prescript{\chart}{}{\Gamma}^\kappa_{\nu\lambda} 
								- \sum_{l} (x_{\kappa,l} \circ x')	x'_{l, \nu \lambda}
						\right)
					\Bigr)_{[\mu\nu]}
			\end{multlined} 
	\\
		&= 
			\begin{multlined}[t]					 
				\smashoperator{\sum_{\substack{i\in I \\ k', l' }}} 
					\partUnity_i 
					\Biggl(\textstyle
				 	\begin{multlined}[c][.7\textwidth]
				 		\textstyle
				 		\sum_{\kappa}
				 		\bigl(  
				 			x'_{k', \kappa \mu } (x^i_{\kappa ,l'} \circ y^i) + x'_{k', \kappa } (x^i_{\kappa ,l'} \circ y^i)_{,\mu} 
				 		\bigr) 
				 	\\
				 		\textstyle
				 		\cdot \sum_{\nu', l}
						\bigl(  
							y^i_{l', l\lambda } (x_{l, \nu'} \circ x') + y^i_{l', l } (x_{l, \nu'} \circ x')_{, \lambda} 
						\bigr)
							 	x'_{\nu', \nu} (x_{k, k'} \circ x') 
				 	\end{multlined}
		 			\Biggr)_{\Helvetica[{[\mu\nu]}]{\mathrlap{[\mu\nu]}}} 
				\\ + M 
				 +
					\sum_{\kappa} \textstyle
					\begin{multlined}[t][7cm]\textstyle
						\bigl(
							\prescript{\chart}{}{\Gamma}^k_{\mu\kappa}   \prescript{\chart}{}{\Gamma}^\kappa_{\nu\lambda}
						+ \sum_{l, l'}
							\prescript{\chart}{}{\Gamma}^k_{\mu\kappa} (x_{\kappa,l} \circ x')	x'_{l, \nu \lambda}
					\\ \textstyle
						 - 	
						\prescript{\chart}{}{\Gamma}^\kappa_{\nu\lambda} (x_{k,l} \circ x')	x'_{l,\mu \kappa}
						+	 (x_{k,l} \circ x')	x'_{l,\mu \kappa}
								\smash{\underbrace{(x_{\kappa,l'} \circ x')	x'_{l', \nu \lambda}}_{
									\text{use formula \cref{eq:eliminateDbDerivative}}
								}}
						\bigr)_{[\mu\nu]}
					\end{multlined}
			\end{multlined} 
	\\[1em]
		&= 
			\begin{multlined}[t]					 
				\smashoperator{\sum_{\substack{i\in I \\ l, \kappa }}} 
					\partUnity_i 
					\Biggl(\textstyle
				 	\begin{multlined}[c][.7\textwidth]
				 		\textstyle
				 		\sum_{k', l'}
				 		\bigl(  
				 			x'_{k', \kappa \mu } (x^i_{\kappa ,l'} \circ y^i) + x'_{k', \kappa } (x^i_{\kappa ,l'} \circ y^i)_{,\mu} 
				 		\bigr) 
				 	\\
				 		\textstyle
				 		\cdot \sum_{\nu}
						\bigl(  
							y^i_{l', l\lambda } (x_{l, \nu'} \circ x') + y^i_{l', l } (x_{l, \nu'} \circ x')_{, \lambda} 
						\bigr)
							 	x'_{\nu', \nu} (x_{k, k'} \circ x') 
				 	\end{multlined}
		 \Biggr)_{[\mu\nu]} 
			\\
				+ 	
					\underbrace{
						\begin{multlined}
						{\textstyle
								M 
							+
								\sum_{\kappa}
								\prescript{\chart}{}{\Gamma}^k_{\mu\kappa}   \prescript{\chart}{}{\Gamma}^\kappa_{\nu\lambda}
							-\sum_{\kappa, l, l'} 
							\prescript{\chart}{}{\Gamma}^k_{\mu\kappa} 
									(x_{\kappa,l} \circ x')	x'_{l, \nu \lambda} }
							\\	+ 
									\prescript{\chart}{}{\Gamma}^\kappa_{\nu\lambda} 
									(x_{k,l} \circ x')	x'_{l,\mu \kappa}
								+	
									(x_{k,l} \circ x')	x'_{l,\mu \kappa} 
									(x_{\kappa,l'} \circ x')_{,\lambda}	x'_{l', \nu}
						\end{multlined}					
					}_{\eqqcolon T }
			\end{multlined}
	\\
		&= 
			\sum_{\substack{i\in I,\\ l, \kappa }}
				\partUnity_i
			 \left(
				 	\begin{gathered}
				 			\textstyle	\sum\limits_{k'}	x'_{k', \kappa \mu }  (x_{k, k'} \circ x') 
				 				\sum\limits_{l'} (x^i_{\kappa ,l'} \circ y^i)	y^i_{l', l\lambda } 
				 				\sum\limits_{\nu'} (x_{l, \nu'} \circ x') x'_{\nu', \nu}  
				 		\\ + \textstyle
				 			\sum\limits_{l'}	(x^i_{\kappa ,l'} \circ y^i)_{,\mu}y^i_{l', l\lambda } 
				 			\sum\limits_{k'}  x'_{k', \kappa } (x_{k, k'} \circ x') 
				 			\sum\limits_{\nu'} (x_{l, \nu'} \circ x')	x'_{\nu', \nu}  
				 		\\ \textstyle	+ 
				 			\sum\limits_{k', \nu'} x'_{k', \kappa \mu }  (x_{l, \nu'} \circ x')_{, \lambda}
				 				x'_{\nu', \nu} (x_{k, k'} \circ x') 
				 			\sum\limits_{l'} (x^i_{\kappa ,l'} \circ y^i)	y^i_{l', l } 
				 		\\ \textstyle	+ 
				 			\sum\limits_{l',\nu'} 
				 				\smash{\underbrace{(x^i_{\kappa ,l'} \circ y^i)_{,\mu} y^i_{l', l }}_{
				 					\text{use formula \cref{eq:eliminateDbDerivative}}}}
				 				\smash{\underbrace{(x_{l, \nu'} \circ x')_{, \lambda}	x'_{\nu', \nu}}_{
				 					\text{use formula \cref{eq:eliminateDbDerivative}}}
				 			\sum\limits_{k'} (x_{k, k'} \circ x') x'_{k', \kappa }}
				 	\end{gathered}
			 \right)_{\mathrlap{[\mu\nu]}} 
			+ 
				T
	\\[.4em]
		&= 
				\sum_{\substack{i\in I \\ l, \kappa }}
						\partUnity_i 
				 \left(\textstyle
					 	\begin{gathered}
					 		\textstyle	\sum_{k'}	x'_{k', \kappa \mu }  (x_{k, k'} \circ x') \prescript{i}{}\Gamma^\kappa_{l\lambda} \delta_{l\nu}  
					 		\textstyle	+ \sum_{l'}	(x^i_{\kappa ,l'} \circ y^i)_{,\mu}y^i_{l', l\lambda } \delta_{\kappa k} \delta_{l\nu}  \\
					 		\textstyle	+ \sum_{k', \nu'} x'_{k', \kappa \mu }  (x_{l, \nu'} \circ x')_{, \lambda}	x'_{\nu', \nu} (x_{k, k'} \circ x') \delta_{\kappa l} \\
					 		\textstyle	+ \sum_{l',\nu'} (x^i_{\kappa ,l'} \circ y^i) y^i_{l', l \mu } (x_{l, \nu'} \circ x')	x'_{\nu', \nu\lambda} \delta_{k \kappa}
					 	\end{gathered}
				 \right)_{\mathrlap{[\mu\nu]}} 
			+ 
				T
	\\
		&=
		\sum_{\substack{i\in I \\ l, \kappa }}
						\partUnity_i 
		 \left(\textstyle
		 	\begin{gathered}
		 		\textstyle	\sum_{k'}	x'_{k', \kappa \mu }  (x_{k, k'} \circ x') \prescript{i}{}\Gamma^\kappa_{\nu\lambda}  
		 		\textstyle	+ \sum_{l'}	(x^i_{k ,l'} \circ y^i)_{,\mu}y^i_{l', \nu\lambda }  \\
		 		\textstyle	+ \sum_{k', \nu'} x'_{k', \kappa \mu }  (x_{\kappa, \nu'} \circ x')_{, \lambda}	x'_{\nu', \nu} (x_{k, k'} \circ x') \\
		 		\textstyle	+ \sum_{\nu'} \prescript{i}{}\Gamma^k_{\mu l} (x_{l, \nu'} \circ x')	x'_{\nu', \nu\lambda}
		 	\end{gathered}
		 \right)_{\mathrlap{[\mu\nu]}} 
				+ 	
					T
	\\
		&= \begin{multlined}[t]
				\sum_{i\in I} \partUnity_i\sum_{l, \kappa }
			 \left( \hspace{-.5cm} 
			 \textstyle
			 	\begin{gathered}
			 		\textstyle	\sum\limits_{k'}	x'_{k', \kappa \mu }  (x_{k, k'} \circ x') \prescript{\chart}{}\Gamma^\kappa_{\nu\lambda}  
			 		\textstyle	+ \sum\limits_{l'}	(x^i_{k ,l'} \circ y^i)_{,\mu}y^i_{l', \nu\lambda }  \\
			 		\textstyle	+ \sum\limits_{k', \nu'} x'_{k', \kappa \mu }  (x_{\kappa, \nu'} \circ x')_{, \lambda}	x'_{\nu', \nu} (x_{k, k'} \circ x') \\
			 		+ \sum\limits_{\nu'} \prescript{\chart}{}\Gamma^k_{\mu l} (x_{l, \nu'} \circ x')	x'_{\nu', \nu\lambda}
			 	\end{gathered}
			 \hspace{-.5cm}
			 \right)_{[\mu\nu]} 
	\\
			 + 	
				\sum_{i\in I}
					\bigl(
						\partUnity_{i, \mu}
							\prescript{i}{}{\Gamma}_{\nu \lambda}^k
					\bigr)_{[\mu\nu]}
				+
					\sum_{\kappa}
					\prescript{\chart}{}{\Gamma}^k_{\mu\kappa}   \prescript{\chart}{}{\Gamma}^\kappa_{\nu\lambda}
			\\
				- 	
					\sum_{\kappa, l, l'} 
						\begin{multlined}[t]
							\prescript{\chart}{}{\Gamma}^k_{\mu\kappa} 
									(x_{\kappa,l} \circ x')	x'_{l, \nu \lambda}
								+ 
									\prescript{\chart}{}{\Gamma}^\kappa_{\nu\lambda} 
									(x_{k,l} \circ x')	x'_{l,\mu \kappa}
							\\
								+	
									(x_{k,l} \circ x')	x'_{l,\mu \kappa} 
									(x_{\kappa,l'} \circ x')_{,\lambda}	x'_{l', \nu}
						\end{multlined}
		\end{multlined}
	\\
		&= 
			\sum_{i\in I} \partUnity_i \sum_{l'}	\bigl( (x^i_{k ,l'} \circ y^i)_{,\mu}y^i_{l', \nu\lambda } \bigr)_{[\mu\nu]}
			+ \bigl(\partUnity_{i, \mu} \prescript{i}{}{\Gamma}_{\nu \lambda}^k\bigr)_{[\mu\nu]}
			+ \sum_{ \kappa} \bigl(\prescript{\chart}{}{\Gamma}^k_{\mu\kappa}   \prescript{\chart}{}{\Gamma}^\kappa_{\nu\lambda}\bigr)_{[\mu\nu]} \\
		&= 
			\prescript{\chart}{}\Riem^k_{\lambda\mu\nu}. 
			\qedhere
	\end{align*}
\end{proof}

\printbibliography

\end{document}